\documentclass[11pt,leqno]{article}
\usepackage[utf8]{inputenc}

\usepackage{amsmath}
\usepackage{amsthm}
\usepackage[shortlabels]{enumitem} \setlist{wide, labelindent=0pt}
\usepackage{mathtools}
\usepackage{stmaryrd}
 
\usepackage[colorlinks=false]{hyperref}

\usepackage{color}

\usepackage[charter]{mathdesign}

\makeatletter
\g@addto@macro\bfseries{\boldmath}
\makeatother

\usepackage{xy}
\xyoption{all}

\numberwithin{equation}{section}

\theoremstyle{plain}
    \newtheorem{theorem}[equation]{Theorem}
    \newtheorem{lemma}[equation]{Lemma}
    
    \newtheorem{proposition}[equation]{Proposition}

    \newtheorem*{theorem*}{Theorem}
    \newtheorem*{proposition*}{Proposition}
    \newtheorem*{corollary*}{Corollary}
    \newtheorem*{lemma*}{Lemma}
    \newtheorem*{conjecture*}{Conjecture}
    \newtheorem{definition-theorem}[equation]{Definition/Theorem}
    \newtheorem{definition-lemma}[equation]{Definition/Lemma}
\theoremstyle{definition}
    \newtheorem{definition}[equation]{Definition}
    \newtheorem{example}[equation]{Example}
    \newtheorem{examples}[equation]{Examples}

    \newtheorem{remark}[equation]{Remark}
    
    \newtheorem{remarks}[equation]{Remarks}

    \newtheorem{exercise}[equation]{Exercise}
    \newtheorem*{exercise*}{Exercise}
    \newtheorem*{solution*}{Solution}
    
   \newtheorem{thm1}{Theorem}
 
    \newcommand{\R}{\mathbb{R}}
    \newcommand{\C}{\mathbb{C}}

   	\renewcommand{\phi}{\varphi}
	\let\epsilon\varepsilon

    \newcommand{\Bounded}{\operatorname{B}}
    
    \newcommand{\Compact}{\operatorname{K}}
    \newcommand{\Adjointable}{\operatorname{L}}
    
    \newcommand{\Unitary}{\operatorname{U}}
    \newcommand{\Mat}{\operatorname{M}}

\newcommand{\llangle}{\langle\!\langle}
\newcommand{\rrangle}{\rangle\!\rangle}
\newcommand{\into}{\hookrightarrow}

\newcommand{\dd}{d}  
\newcommand{\restrict}{\raisebox{-.5ex}{$|$}}

\newcommand{\id}{\mathrm{id}}

  \renewcommand{\prod}{\bigsqcap}

    \DeclareMathOperator{\Ind}{Ind}

    \DeclareMathOperator{\lspan}{span}
	
    \DeclareMathOperator{\SL}{SL}
    \DeclareMathOperator{\GL}{GL}
    
    \DeclareMathOperator{\trace}{trace}

    \DeclareMathOperator{\HS}{HS}

\newcommand{\germ}{\mathfrak}

\newcommand{\bra}[1]{{\langle{#1}\vert}}
	\newcommand{\ket}[1]{{\vert {#1}\rangle}}

      \newcommand{\bmat}[1]{{\ensuremath \begin{bmatrix} #1 \end{bmatrix}}}
\newcommand{\smallbmat}[1]{{\ensuremath \left[\begin{smallmatrix} #1 \end{smallmatrix}\right]}}

\DeclareMathOperator{\ev}{ev}

\title{Notes on $C^*$-algebras, representations, and Morita equivalence \\ \medskip \Large with a view toward $C^*$-algebras of reductive groups}
\author{Tyrone Crisp\thanks{Department of Mathematics \& Statistics, University of Maine. \href{mailto:tyrone.crisp@maine.edu}{tyrone.crisp@maine.edu}}}
\date{November 2025; revised February 2026}
 
\begin{document}

\maketitle

These notes give an expanded account of my lectures at the CIRM-IHP research school on \emph{Methods in representation theory and operator algebras}, January 6--10, 2025. Their goal is to explain a proof of the following theorem:

\begin{thm1}\label{thm:main}
For each real reductive group $G$ there is a Morita equivalence of $C^*$-algebras 
\[
C^*_r(G)\sim 
\bigoplus_{[L,\sigma]} C_0(\germ{a}_L^*/W'_\sigma) \rtimes R_\sigma.
\]
\end{thm1}
\noindent(The meaning of the notation will be explained in due course.) 

This theorem first appeared in a short note by A.~Wassermann \cite{Wassermann} (referring to another short note by Arthur \cite{Arthur}), with the case of complex groups having been worked out earlier by Penington--Plymen \cite{Penington-Plymen}. Wassermann used this description of the reduced $C^*$-algebra of a real reductive group to verify (in the linear reductive case) the Connes--Kasparov conjecture, which computes the $K$-theory of these $C^*$-algebras in index-theoretic terms. Our understanding of the connections between representation theory and $C^*$-algebras inherent in the Connes--Kasparov isomorphism continues to advance: see for instance \cite{BCH,Lafforgue-ICM,Higson-Mackey,CCH,Afgoustidis-Mackey,CHST,CHS,BHY,CHR}, among other works.

These notes, like the lectures that they document, are intended as an on-ramp for newcomers to this area. They do not attempt to provide a comprehensive survey of the field; rather, we have tried to map out a reasonably accessible route to one specific result---Theorem \ref{thm:main}---while introducing some of the basic ideas from $C^*$-algebra theory that are used in this computation. Our target audience is Masters- and PhD-level students, and other mathematicians who are not specialists in $C^*$-algebras.  The main prerequisites are some knowledge of basic functional analysis, and a willingness to accept without proof the basic results about $C^*$-algebras that we will rely on. (Some prior exposure to the representation theory of finite or compact groups would also be helpful, though we have summarised the necessary facts in  an appendix.) 

Our tight focus on Theorem \ref{thm:main} has meant glossing over  many fundamental ideas about $C^*$-algebras, while discussing comparatively simple results in much greater detail. To readers looking for a more comprehensive introduction to the theory of $C^*$-algebras and their representations, we warmly recommend Dixmier's classic \cite{Dixmier-Cstar-fr,Dixmier-Cstar-en}, and Blackadar's survey \cite{Blackadar}. For an historical overview of some other connections between $C^*$-algebras and group representations, see Rosenberg's survey \cite{Rosenberg}. For Hilbert modules and Morita equivalence see Rieffel's survey \cite{Rieffel-survey}, and the books by Raeburn and Williams \cite{Raeburn-Williams} and Lance \cite{Lance}.

These notes concentrate almost exclusively on the $C^*$-algebraic aspects of Theorem \ref{thm:main}: owing both to a lack of time during the lectures, and a lack of expertise, we have not attempted to explain the many deep results in the representation theory of real reductive groups (due to Harish-Chandra, Langlands, Knapp, Stein, and others) that our $C^*$-algebraic computation relies on. For further details about how these representation-theory results are used to study $C^*$-algebras, see  for instance \cite{CCH} and \cite{CHST}. For a fuller treatment of the field, including references to the original sources, see for instance the books of Knapp \cite{Knapp-overview} and Wallach \cite{WallachI,WallachII}.

 The main expository sections of these notes are Section \ref{sec:Cstar}, which discusses the basic theory of $C^*$-algebras and their representations, culminating in Kaplansky's Stone--Weierstrass theorem; and Section \ref{sec:Morita}, on the notion of Morita equivalence for $C^*$-algebras. Each of these longer sections is followed by a shorter one in which we briefly explain how the general theory can be applied to the $C^*$-algebras of real reductive groups, giving a proof of Theorem \ref{thm:main}. In addition to Wassermann's brief note \cite{Wassermann}, proofs of (parts of) Theorem \ref{thm:main} can be found in \cite{CCH}, \cite{AA}, and \cite{CHST}, and our exposition here is indebted to all of those sources, as well as to  work in progress with Pierre Clare.  
 
We have also included an appendix listing some basic facts from the representation theory of finite and compact groups (Appendix \ref{sec:compact-groups}), as well as solutions---some of them rather sketchy, and none of them especially elegant---to the exercises that are scattered throughout the notes (Appendix \ref{sec:solutions}.) The exercises are for the most part fairly routine verifications, whose purpose is to gently introduce students to some of the basic computations and lines of argument in this area.

\paragraph*{Acknowledgements:} It is a pleasure to thank the organisers Alexandre Afgoustidis, Anne-Marie Aubert, Pierre Clare, and Haluk \c{S}eng\"un, for having invited me to give these lectures; and to thank all of the participants, and our hosts at CIRM, for a stimulating and very enjoyable meeting. My own participation in this event was supported by funding from the American Mathematical Society, the Simons Foundation, and the Universit\'e de Lorraine.

\section{Representations of $C^*$-algebras}\label{sec:Cstar}

Our goal in these notes, loosely speaking, is to identify the (reduced) $C^*$-algebra of a real reductive group with an algebra of functions on a space of representations, via a kind of Fourier transform. An important step in this computation will be to determine the range of our Fourier transform. In general, the determination of the range of a Fourier transform on a given space of functions relies on a sometimes delicate interplay between local regularity properties and asymptotic behaviour. In our $C^*$-algebraic setting the picture is quite straightforward: the relevant local regularity property is \emph{continuity}, and the relevant asymptotic behaviour is \emph{vanishing at infinity}. In this section we will introduce the tool that we will use to prove this fact about the range of the Fourier transform: a generalisation of the Stone--Weierstrass theorem, due  to Kaplansky.

Before we get to that,  we will introduce the basic ideas and results about $C^*$-algebras and their representations that are needed to understand what the Stone--Weierstrass theorem says, and why it is true.

\subsection{$C^*$-algebras: basic definitions and examples}

\begin{definition}
\begin{enumerate}[\rm(1)]
\item A \emph{$C^*$-algebra} is a Banach space (over $\C$) with an associative bilinear multiplication satisfying $\|ab\|\leq \|a\|\,\|b\|$, equipped with an involution $*:A\to A$ that is conjugate-linear, satisfies $(ab)^*=b^*a^*$, and is related to the norm by the identity $\| a^* a\| = \|a\|^2$.
\item A \emph{homomorphism} of $C^*$-algebras $\phi:A\to B$ is a linear map compatible with multiplication and involution: $\phi(ab)=\phi(a)\phi(b)$ and $\phi(a^*)=\phi(a)^*$ for all $a,b\in A$.
\end{enumerate}
\end{definition}

\begin{remarks}
\begin{enumerate}[\rm(1)]
\item A $C^*$-algebra is not required to possess a multiplicative identity. Almost none of the $C^*$-algebras that we will meet in these notes have one. But for many purposes $C^*$-algebras behave like algebras that have $1$; for instance, the multiplication map $A\times A\to A$ is surjective (\emph{cf}.~Theorem \ref{thm:Cohen-Hewitt}.)
\item Every $C^*$-algebra has a norm, and therefore a topology, but the definition of homomorphisms of $C^*$-algebras does not require continuity, because the norm is completely determined by the algebraic structure: {\bf Theorem:} Homomorphisms of $C^*$-algebras are \emph{automatically} contractive with closed range; and injective homomorphisms are automatically isometric. 
\item The point above implies that if an algebra with involution has a norm making it a $C^*$-algebra, then there is only one such norm. For this reason the norm often recedes into the background when we study $C^*$-algebras: the norm can be very useful (because we can use completeness to build a solution to a problem by building a sequence of progressively better approximate solutions); but at other times we don't really need to think about the norm.
\item For more or less the same reason, the $*$ operation on a $C^*$-algebra is automatically isometric: $\|a^*\|=\|a\|$ for all $a\in A$. 
\end{enumerate}
\end{remarks}

\begin{examples}\label{examples-Cstar}
\begin{enumerate}[\rm(1)]
\item Let $X$ be a locally compact Hausdorff topological space, and consider the set $C_0(X)$ of continuous functions $f:X\to \C$ that vanish at infinity (\emph{i.e.}: for every $\epsilon>0$ the set $\{x\in X\ |\ |f(x)|\geq \epsilon\}$ is compact.) We add and multiply functions pointwise; $\|f\|\coloneqq \sup_{x\in X}|f(x)|$; and $f^*(x)\coloneqq \overline{f(x)}$. This makes $C_0(X)$ into a $C^*$-algebra. {\bf Theorem:} Every commutative $C^*$-algebra is isomorphic to one of this form.
\item Let $H$ be a complex Hilbert space, and let $\Bounded(H)$ be the space of all bounded operators $H\to H$. We multiply operators by composing them. We let $*$ be the usual adjoint operation, characterised by $\langle \xi\,|\,t^*(\eta)\rangle = \langle t(\xi)\, |\, \eta\rangle$ for all $\xi,\eta\in H$. Equipped with the operator norm ($\|t\| = \sup_{\|\xi\|=1}\|t\xi\|$), $\Bounded(H)$ is a $C^*$-algebra. (\emph{Convention:} throughout these notes our inner products $\langle\xi\, |\, \eta\rangle$ are linear in $\eta$ and conjugate-linear in $\xi$.)
\item {\bf Definition:} By a \emph{subalgebra} of a $C^*$-algebra we mean a norm-closed linear subspace that is closed under multiplication and $*$. Such a subalgebra is itself a $C^*$-algebra. {\bf Theorem:} Every $C^*$-algebra is isomorphic to a subalgebra of $\Bounded(H)$ for some $H$. 
\item The space of \emph{compact} operators $\Compact(H)$, on the Hilbert space $H$, is a subalgebra of $\Bounded(H)$, hence a $C^*$-algebra. 
\item Combining examples (1) and (4), if $X$ is a locally compact Hausdorff space, and $H$ is a Hilbert space, we consider the space $C_0(X,\Compact(H))$ of continuous functions $f:X\to \Compact(H)$ that vanish (with respect to the norm on $H$) at infinity. Addition, scalar multiplication, multiplication, and the $*$ are all defined pointwise, and the norm is $\|f\| \coloneqq \sup_{x\in X}\|f(x)\|$. This makes $C_0(X,\Compact(H))$ into a $C^*$-algebra.
\item Let $G$ be a locally compact group. All integrals on $G$ are taken with respect to some choice of left-invariant Haar measure. The convolution algebra $C_c(G)$ of continuous, compactly supported, complex-valued functions on $G$ acts as bounded operators on the Hilbert space $L^2(G)$: for $f\in C_c(G)$ and $\xi\in L^2(G)$ we define
\[
\lambda(f)\xi (g) = \int_G f(h)\xi(h^{-1}g)\, \dd h.
\]
Then we define $C^*_r(G) = \overline{\lambda(C_c(G))}$, the closure of $\lambda(C_c(G))$ in the operator norm. This is a $C^*$-algebra, called the \emph{reduced $C^*$-algebra} of the group $G$.
\item There are other $C^*$-algebras attached to locally compact groups, obtained by completing $C_c(G)$ in different norms. In these notes we will only talk about reduced $C^*$-algebras.
\item Let $W$ be a finite group acting on a $C^*$-algebra $A$ by $*$-automorphisms: this means that for each $w\in W$ we have an invertible linear map $\beta_w:A\to A$ satisfying $\beta_w(ab)=\beta_w(a)\beta_w(b)$ and $\beta_w(a^*)=\beta_w(a)^*$; and these maps satisfy $\beta_{w_1}\circ \beta_{w_2}=\beta_{w_1 w_2}$. We form two new $C^*$-algebras from these data:
\begin{enumerate}[\rm(i)]
\item The \emph{fixed-point algebra} 
\[
A^W = \{a\in A\ |\ \beta_w(a)=a\text{ for all }w\in W\}.
\]
This is a subalgebra of $A$, and thus a $C^*$-algebra with the operations inherited from $A$.
\item The \emph{crossed product}
\[
A\rtimes W = \left\{\textstyle\sum_{w\in W} a_w w\ \middle|\ a_w\in A\right\}
\]
is made into an algebra by multiplying elements of $A$ together as in $A$; multiplying elements of $W$ together as in $W$; and setting $wa = \beta_w(a)w$ for $w\in W$ and $a\in A$.  We define an involution $*$ on $A\rtimes W$ by setting $(aw)^*=w^{-1}a^*=\beta_{w^{-1}}(a^*) w^{-1}$. Finally, we need to put a $C^*$-algebra norm on $A\rtimes W$. The details of how this can be done will not be important until later, so we postpone our discussion of this point to Exercise \ref{ex:crossed-product-norm}.
\end{enumerate}
\end{enumerate}
\end{examples}

\begin{exercise}\label{ex:W-acting-on-X}
Let $W$ be a finite group acting by homeomorphisms on a locally compact Hausdorff space $X$. For each $w\in W$ and $f\in C_0(X)$ define $\beta_w(f)(x)=f(w^{-1}x)$. Prove that this defines an action of $W$ on $C_0(X)$ by $*$-automorphisms, and that $C_0(X)^W \cong C_0(X/W)$.
\end{exercise}

\begin{exercise}\label{ex:iterated-crossed-product}
Let $W$ be a finite group, acting on a $C^*$-algebra $A$ by $*$-automorphisms. Suppose that $W$ is a semi-direct product $U\rtimes V$. Prove that the formula $\alpha_v(au)\coloneqq \beta_v(a) vuv^{-1}$ defines an action of $V$ on the crossed product $A\rtimes U$, and find an isomorphism of $C^*$-algebras $(A\rtimes U)\rtimes V \cong A\rtimes W$. 
\end{exercise}

\begin{definition}
By an \emph{ideal} in a $C^*$-algebra $A$ we mean a two-sided ideal that is closed in the norm topology.
\end{definition}

\begin{theorem}
Let $J$ be an ideal in $A$.
\begin{enumerate}[\rm(1)]
\item We have $j^*\in J$ for all $j\in J$, and so $J$ is a $C^*$-algebra.
\item The quotient $A/J$ is a $C^*$-algebra, under the quotient norm
\[
\|a+J\| = \inf_{j\in J}\|a+j\|.
\]
\item All of the expected isomorphism theorems are valid.\hfill\qed
\end{enumerate}
\end{theorem}

\begin{example}
Let $U$ be an open subset of a locally compact Hausdorff space $X$. Restriction of functions from $X$ to the closed subset $X\setminus U$ defines a homomorphism of $C^*$-algebras $C_0(X)\to C_0(X\setminus U)$. This homomorphism is surjective (by Tietze's extension theorem), and its kernel is the ideal $C_0(U)\subseteq C_0(X)$. Thus $C_0(X)/C_0(U)\cong C_0(X\setminus U)$. The map $U\mapsto C_0(U)$ is a bijection between the set of open subsets of $X$, and the set of ideals of $C_0(X)$.
\end{example}

\begin{example}
The algebra of compact operators $\Compact(H)$ is an ideal in $\Bounded(H)$. The algebra $\Compact(H)$ itself is \emph{simple}: its only ideals are $0$ and $\Compact(H)$.
\end{example}

\subsection{Representations of $C^*$-algebras}

\begin{definition}\label{def:Cstar-reps}
Let $A$ be a $C^*$-algebra, and let $H$ be a Hilbert space.
\begin{enumerate}[\rm(1)]
\item A \emph{representation} of $A$ on $H$ is a homomorphism $\pi : A\to \Bounded(H)$.
\item A representation $\pi:A\to \Bounded(H)$ is called \emph{irreducible} if $H\neq 0$ and $H$ has no proper, nonzero, $A$-invariant closed subspaces. (A theorem of Kadison says that this is equivalent to the condition that $H$ has no proper nonzero $A$-invariant subspaces at all.)
\item Two irreducible representations $\pi_1:A\to \Bounded(H_1)$ and $\pi_2:A\to\Bounded(H_2)$ are \emph{equivalent}, or \emph{isomorphic}, if there is a unitary operator $u:H_1\to H_2$ such that $u\pi_1(a)=\pi_2(a)u$ for all $a\in A$. 
\item The set $\widehat{A}$ of equivalence classes of irreducible representations of $A$ carries a natural topology, in which the open sets are in bijection with the ideals in $A$: to each ideal $J$ of $A$ we associate the open set
\[
O_J = \left\{ \pi\in \widehat{A}\ \middle|\ \pi\restrict_J\neq 0 \right\}.
\]
(We usually won't distinguish in the notation between a representation $\pi$ and its equivalence class in $\widehat{A}$. Note that equivalent representations have equal kernels.)
\end{enumerate}
\end{definition}

The following fact is often useful when studying representations of $C^*$-algebras:

\begin{theorem}[Cohen--Hewitt factorisation]\label{thm:Cohen-Hewitt}
Let $A$ be a $C^*$-algebra, and let $X$ be a left Banach $A$-module (\emph{i.e.}, $X$ is a Banach space and a left $A$-module, and $\|a x \|\leq \|a\| \|x\|$ for all $a\in A$ and $x\in X$.) The set $AX=\{ax\ |\ a\in A,\ x\in X\}$ is a closed $A$-submodule of $X$.\hfill\qed
\end{theorem}

See, eg, \cite[Proposition 2.33]{Raeburn-Williams} for an exposition of a proof of Theorem \ref{thm:Cohen-Hewitt} due to Blanchard. The theorem implies that if $\pi:A\to \Bounded(H)$ is an irreducible representation, then $H=\{\pi(a)\xi\ |\ a\in A,\ \xi\in H\}$. (The theorem also applies to right modules, with the obvious modifications.)

Several important facts about $C^*$-algebra representations follow from the close relationship between representations and \emph{states}:

\begin{definition}\label{def:states}
Let $A$ be a $C^*$-algebra.
\begin{enumerate}[\rm(1)]
\item We say that $a\in A$ is  \emph{positive} (notation: $a\geq 0$) if $a=b^*b$ for some $b\in A$. ({\bf Theorem:} $a\geq 0$ if and only if for every representation $\pi:A\to \Bounded(H)$ and every $\xi\in H$ we have $\langle \xi\ |\ \pi(a)\xi\rangle\geq 0$.) 
\item A \emph{state} on $A$ is a linear map $\phi:A\to \C$ with $\|\phi\|=1$, such that $\phi(a)\geq 0$ whenever $a\geq 0$. 
\item Given a state $\phi$ on $A$, we construct a representation $\pi_\phi : A\to \Bounded(H_\phi)$ as follows:
\begin{enumerate}[$\bullet$]
\item Let $J_\phi \coloneqq \{a\in A\ |\ \phi(a^*a)=0\}$. 
\item $H_\phi \coloneqq \overline{ A/J_\phi}$, the completion in the norm $\|a+J_\phi\|=\phi(a^*a)^{1/2}$.
\item $H_\phi$ is a Hilbert space, with inner product $\langle a+J_\phi\ |\ b+J_\phi\rangle = \phi(a^*b)$.
\item $\pi_\phi:A\to  \Bounded(H_\phi)$ is defined by $\pi_\phi(a)(b+J_\phi) = (ab)+J_\phi$.
\end{enumerate}
This is the  \emph{GNS construction}, named after Gelfand, Naimark, and Segal. 
\end{enumerate}
\end{definition}

This construction, together with the Hahn--Banach and Krein--Milman theorems, leads to: 

\begin{theorem}\label{thm:reps}
Let $A$ be a $C^*$-algebra.
\begin{enumerate}[\rm(1)]
\item A representation of $A$ is irreducible if and only if it is isomorphic to the GNS representation associated to a \emph{pure} state of $A$: that is, a state that is not a convex combination of other states.
\item The irreducible representations of $A$ separate points: that is, if $a\neq b$ in $A$ then there is an irreducible representation $\pi$ of $A$ with $\pi(a)\neq \pi(b)$.
\item Let $B$ be a subalgebra of a $C^*$-algebra $A$, and let $\pi:B\to \Bounded(H)$ be an irreducible representation. There is an irreducible representation $\pi':A\to \Bounded(H')$, where $H'$ is a Hilbert space containing $H$ as a $B$-invariant closed subspace, such that $\pi(b)\xi =\pi'(b)\xi$ for all $b\in B$ and $\xi\in H$.\hfill\qed
\end{enumerate}
\end{theorem}

\subsection{Aside: the  bicommutant theorem and irreducibility}

There is a characterisation of irreducible representations that relies on ideas from the theory of \emph{von Neumann algebras}. Von Neumann algebras will not be a focus of these notes, but the irreducibility condition will be very useful, so we briefly recall:

\begin{definition}
Let $H$ be a Hilbert space, and let $\Bounded(H)$ be the $C^*$-algebra of bounded linear operators on $H$.
\begin{enumerate}[\rm(1)]
\item The \emph{strong-operator topology} on $\Bounded(H)$ is the topology in which a net of operators $(t_i)_{i\in I}$ converges to an operator $t$ if and only if $t_i\xi$ converges to $t\xi$, in the norm topology on $H$, for every $\xi\in H$.
\item A \emph{von Neumann algebra} is a $C^*$-subalgebra of $\Bounded(H)$ that contains $\id_H$ and is closed in the strong operator topology.
\end{enumerate}
\end{definition}

Somewhat like the norm topology on a $C^*$-algebra, the condition of strong-operator closure has an algebraic aspect: 

\begin{theorem}[Von Neumann]
A $C^*$-algebra $C\subseteq \Bounded(H)$ is a von Neumann algebra if and only if $C=C''$, where the \emph{commutant} $X'$ of a set of operators $X$ is defined as
\[
X' = \{t\in \Bounded(H)\ |\ tx=xt\text{ for all }x\in X\},
\]
and the \emph{bicommutant} $X''$ is defined as $(X')'$.\hfill\qed
\end{theorem}

It is tautologically true that $X'=X'''$, and so the commutant $C'$ of a $C^*$-algebra $C\subseteq \Bounded(H)$ is a von Neumann algebra. See \cite{Dixmier-vN}, for example, for more information about von Neumann algebras.

Here is how we will use this theorem:

\begin{theorem}\label{thm:SvN}
A representation $\pi:A\to \Bounded(H)$ of a $C^*$-algebra $A$ is irreducible if and only if one of the following equivalent conditions holds:
\begin{enumerate}[\rm(1)]
\item $\pi(A)' = \C \id_H$.
\item $\pi(A)$ is dense in $\Bounded(H)$ with respect to the strong-operator topology.
\end{enumerate}
\end{theorem}

\begin{proof}[Proof (outline)] 
The equivalence of (1) and (2) follows immediately from the bicommutant theorem, because $(\C \id_H)'= \Bounded(H)$ and $\Bounded(H)'=\C \id_H$. The equivalence between condition (1) and the irreducibility of $\pi$ is a version of Schur's lemma. It follows from the observation that the closed invariant subspaces for $\pi$ are in one-to-one correspondence with the orthogonal projections in $\pi(A)'$, and from the fact that every von Neumann algebra is densely spanned by its projections. (The latter fact follows from the spectral theorem for self-adjoint bounded operators: every such operator $c$ can be approximated by finite linear combinations of spectral projections; and those spectral projections are strong-operator limits of polynomials in $c$.)
\end{proof}

\subsection{Representations and ideals}

The relationship between the representations of a $C^*$-algebra $A$, and the representations of the ideals and quotients of $A$, is rather straightforward:

\begin{theorem}\label{thm:ideal-irreps}
Let $J$ be an ideal in a $C^*$-algebra $A$.
\begin{enumerate}[\rm(1)]
\item If $\pi:A\to \Bounded(H)$ is an irreducible representation of $A$, then the restriction $\pi\restrict_J:J\to \Bounded(H)$ is either an irreducible representation, or the zero map.
\item If $\rho:J\to \Bounded(H)$ is an irreducible representation, then there is an irreducible representation $\pi:A\to \Bounded(H)$ with $\rho=\pi\restrict_J$.
\end{enumerate}
Thus the irreducible representations of $J$ are precisely the restrictions to $J$ of those irreducible representations of $A$ that do not annihilate $J$.
\begin{enumerate}
\item[\rm(3)] If $\pi$ and $\rho$ are irreducible representations of $A$ that do not annihilate $J$, then $\pi$ and $\rho$ are equivalent as representations of $A$ if and only if they are equivalent as representations of $J$.
\item[\rm(4)] The map $\pi\mapsto \pi\restrict_J$ gives a homeomorphism $O_J\xrightarrow{\cong}\widehat{J}$.
\item[\rm(5)] If $\pi:A\to \Bounded(H)$ is an irreducible representation that \emph{does} annihilate $J$, then the map $\widetilde{\pi}:A/J\to \Bounded(H)$, $a+J\mapsto \pi(a)$ is an irreducible representation of $A/J$; every irreducible representation of $A/J$ has this form; and two irreducible representations $\pi,\rho$ of $A$ that annihilate $J$ are equivalent if and only if $\widetilde{\pi}$ and $\widetilde{\rho}$ are equivalent representations of $A/J$.
\item[\rm(6)] The map $\pi\mapsto\widetilde{\pi}$ gives a homeomorphism $\widehat{A}\setminus O_J \xrightarrow{\cong} \widehat{A/J}$.
\item[\rm(7)] $J=A$ if and only if no irreducible representation of $A$ annihilates $J$.
\item[\rm(8)] If $J\neq A$ then $\displaystyle J = \bigcap_{\pi\in \widehat{A},\,  \pi(J)=0} \ker \pi$.
\end{enumerate}
\end{theorem}

\begin{exercise}\label{ex:representation-proofs}
Use the results about representations of $C^*$-algebras that have been cited previously to prove Theorem \ref{thm:ideal-irreps}.
\end{exercise}

\subsection{Irreducible representations of $C^*_r(G)$}

\begin{definition}
\begin{enumerate}[\rm(1)]
\item Let $G$ be a locally compact topological group. A \emph{unitary representation} of $G$ is a homomorphism $\pi:G\to \Unitary(H)$ from $G$ to the group of unitary operators on some Hilbert space $H$, that is continuous for the strong-operator topology on $\Unitary(H)$. (That is, for each $\xi\in H$ the function $G\to H$, $g\mapsto \pi(g)\xi$ should be continuous with respect to the norm on $H$.)
\item A unitary representation $\pi:G\to \Unitary(H)$ is \emph{irreducible} if $H$ has no nonzero, proper, $G$-invariant closed subspaces.
\item Two unitary representations $\pi_i:G\to \Unitary(H_i)$ ($i=1,2$) are \emph{equivalent} if there is a unitary operator $u:H_1\to H_2$ satisfying $u\pi_1(g)=\pi_2(g)u$ for all $g\in G$.
\item The set of equivalence classes of irreducible unitary representations of $G$ is denoted by $\widehat{G}$.
\end{enumerate}
\end{definition}

\begin{example}
For any group $G$ we have the \emph{left-regular representation} $\lambda:G\to\Unitary(L^2(G))$ (where the $L^2$ space is with respect to a left-invariant Haar measure), given by $(\lambda(g)\xi)(h)=\xi(g^{-1}h)$ for $\xi\in L^2 (G)$ and $g,h\in G$.
\end{example}

For each unitary representation $\pi:G\to \Unitary(H)$, and each $f\in C_c(G)$, the function
\[
F(\xi,\eta)= \int_G \langle \xi\ |\ f(g)\pi(g)\eta\rangle\, \dd g
\]
defines a bounded sesquilinear form on $H$, and so the Riesz representation theorem implies that there is an operator $\widetilde{\pi}(f)\in \Bounded(H)$ with $F(\xi,\eta)=\langle\xi\, |\, \widetilde{\pi}(f)\eta\rangle$. A straightforward integral-juggling argument shows that $\widetilde{\pi}(f_1\ast f_2) = \widetilde{\pi}(f_1)\widetilde{\pi}(f_2)$, where $\ast$ denotes convolution with respect to Haar measure. (See \cite[13.3]{Dixmier-Cstar-fr}, for instance, for details.) The map $\widetilde{\pi}:C_c(G)\to \Bounded(H)$  is sometimes called the \emph{integrated form} of $\pi$. We usually drop the $\sim$ and just denote both maps by the same symbol, $\pi$. We have already seen an example of this: we write $\lambda$ both for the regular representation $G\to \Unitary(L^2(G))$, and for the induced map $C_c(G)\to \Bounded(L^2(G))$ sending $f\in C_c(G)$ to the convolution operator $\lambda(f)$.

\begin{definition}
We denote by $\widehat{G}_r$ the set of equivalence classes of irreducible unitary representations $\pi$ of $G$ with the property that $\|\widetilde{\pi}(f)\|\leq \|\lambda(f)\|$ for all $f\in C_c(G)$.
\end{definition}

\begin{examples}
\begin{enumerate}[\rm(1)]
\item Let $G$ be a compact group. The Schur orthogonality relations imply that every irreducible representation of $G$ is isomorphic to some $G$-invariant subspace of the regular representation $L^2(G)$. (See Appendix \ref{sec:compact-groups} for reminders about representations of compact groups.) Thus for each $\pi\in \widehat{G}$, and each $f\in C(G)$, the operator $\pi(f)$ is (up to isomorphism) the restriction of the operator $\lambda(f)$ to an invariant subspace, and so $\|\pi(f)\|\leq \|\lambda(f)\|$. Thus $\widehat{G}=\widehat{G}_r$.
\item Let $G=\R$ with addition. Schur's lemma implies that every irreducible representation of $\R$ is one-dimensional, \emph{i.e.}~a continuous homomorphism $\R\to \Unitary(\C)$, and it is not difficult to show that every such homomorphism has the form $\pi_y:x\mapsto e^{ixy}$ for some $y\in \R$. The integrated form of $\pi_y$ is the map $\pi_y:C_c(\R)\to \C$, $\pi_y(f) = \int_{\R} f(x)e^{ixy}\, \dd x$, which we recognise as (one possible normalisation of) the Fourier transform $\hat{f}(y)$. On the other hand, $\lambda(f)$ is the operator on $L^2(\R)$ given by convolution with $f$. Conjugating by the (unitary) Fourier transform on $L^2(\R)$ transforms this operator into that of pointwise multiplication by $\widehat{f}$, whose norm is $\sup_{x\in \R}|\hat{f}(x)|$. So $|\pi_y(f)|\leq \|\lambda(f)\|$, and we conclude that $\pi_y\in \widehat{\R}_r$ for every $y\in \R$. Thus $\widehat{\R}=\widehat{\R}_r$. The same argument applies to every abelian group.
\item For most real reductive groups we have $\widehat{G}\neq \widehat{G}_r$. For example, the trivial one-dimensional representation $\SL(2,\R)\to \Unitary(\C)$ is not in $\widehat{\SL(2,\R)}_r$. In these notes we will not attempt to explain why, as a full discussion of this topic (\emph{amenability}) would require its own course of lectures.  Instead, we will just quote a theorem, due to Harish-Chandra and Cowling--Haagerup--Howe \cite{CHH}, that identifies the set $\widehat{G}_r$, for the groups $G$ that are our main focus in these notes.
\end{enumerate}
\end{examples}

\begin{theorem}\label{thm:tempered}
Let $G$ be a real reductive group. The irreducible unitary representations of $G$ that lie in $\widehat{G}_r$ are precisely the \emph{tempered} representations, and these are precisely those whose matrix-coefficient functions $g\mapsto \langle \pi(g)\xi\, |\, \eta\rangle$ are of class $L^{2+\epsilon}$ modulo the centre of $G$.
\end{theorem}

Here is why $\widehat{G}_r$ is important for our computation of $C^*_r(G)$ (and vice versa): 

\begin{lemma}\label{lem:Ghat-CstarGhat}
Let $G$ be a locally compact group. The map sending each irreducible unitary representation $\pi$ of $G$ to the integrated form $\pi:C_c(G)\to \Bounded(H_\pi)$ induces a bijection $\widehat{G}_r \cong \widehat{C^*_r(G)}$.
\end{lemma}

\begin{proof}[Proof (outline)]
The definition of $\widehat{G}_r$ ensures that the integrated form of an irreducible unitary representation $\pi$ extends by continuity to $C^*_r(G)$ if and only if $\pi\in \widehat{G}_r$. An approximation argument shows that for unitary representations $\pi_1,\pi_2$ of $G$,  a unitary operator $u:H_1\to H_2$ satisfies $u\pi_1(g)=\pi_2(g)u$ for all $g\in G$ if and only if the integrated forms  satisfy $u \pi_1(f)=\pi_2(f) u$ for every $f\in C_c(G)$. Thus the map sending each irreducible unitary representation of $G$ to its integrated form induces a well-defined injective map $\widehat{G}_r\to \widehat{C^*_r(G)}$. To prove that this map is surjective we first observe that the left-translation action $\lambda$ of $G$ on $C_c(G)$ is isometric for the norm $\|\lambda(f)\|$, and therefore extends to give an isometric action of $G$ on $C^*_r(G)$ (not by $C^*$-algebra automorphisms, but by so-called \emph{unitary multipliers}). Now, given an irreducible representation $\pi:C^*_r(G)\to\Bounded(H)$, we obtain a unitary representation $G\to \Unitary(H)$ whose integrated form is $\pi$ by identifying $H = \{\pi(f)\xi\ |\ f\in C^*_r(G),\, \xi\in H\}$ (Theorem \ref{thm:Cohen-Hewitt}), and then setting $\pi(g)\left(\pi(f)\xi\right)\coloneqq \pi(\lambda(g)f)\xi$. 
\end{proof}

\begin{remark}
The topology on $\widehat{C^*_r(G)}$ described in Definition \ref{def:Cstar-reps} induces, via the bijection $\widehat{G}_r\cong \widehat{C^*_r(G)}$, a topology on $\widehat{G}_r$. This topology has a very natural description: a net of (equivalence classes of) irreducible unitary representations $\pi_i:G\to \Unitary(H_i)$ converges to an irreducible representation $\pi:G\to \Unitary(H)$ if there are unit vectors $\xi_i\in H_i$ and $\xi\in H$ such that the net of matrix-coefficient functions $g\mapsto \langle \pi(g)\xi_i\, |\, \xi_i\rangle$ converges, uniformly on compact subsets of $G$, to the function $g\mapsto \langle  \pi(g)\xi\, |\,\xi\rangle$. See \cite[18.1]{Dixmier-Cstar-fr} for details.
\end{remark}

Our computation of $C^*_r(G)$ for a real reductive group $G$ works by using theorems about tempered representations of $G$ to match  up the irreducible representations of $C^*_r(G)$ with the irreducible representations of other, simpler-looking $C^*$-algebras. We will now build up the results about those simpler-looking algebras that we shall need.

\subsection{Irreducible representations of $C_0(X)$}

Schur's lemma implies that every irreducible representation $\pi$ of a commutative $C^*$-algebra $C_0(X)$ is one-dimensional, \emph{i.e.}, a homomorphism of $C^*$-algebras $\pi:C_0(X)\to \C$. The Riesz representation theorem implies that $\pi$ is integration with respect to some measure, and the multiplicative property $\pi(fg)=\pi(f)\pi(g)$ implies that this measure must be evaluation at a single point: thus $\pi(f)=\ev_x(f)\coloneqq f(x)$ for some $x\in X$. If the representations $\ev_x$ and $\ev_y$ are equivalent then they have the same kernel. If $x\neq y$ then there is a function $f\in C_0(X)$ with $f(x)=0\neq f(y)$, so $\ev_x\not\cong \ev_y$. Thus the map $x\mapsto \ev_x$ is a bijection from $X$ to $\widehat{C_0(X)}$. With a little more work, one can prove:

\begin{theorem}\label{thm:CX-irreps}
The map $x\mapsto \ev_x$ is a homeomorphism $X\to \widehat{C_0(X)}$.
\end{theorem}

\begin{exercise}\label{ex:CX-spectrum}
We proved above that the map $x\mapsto \ev_x$ is bijective. Now prove that it is a homeomorphism.
\end{exercise}

\subsection{Irreducible representations of $\Compact(H)$}

\begin{exercise}\label{ex:K-irreducible}
If $H$ is a nonzero Hilbert space, then the identity representation $\Compact(H)\to \Bounded(H)$ is irreducible.
\end{exercise}

\begin{theorem}\label{thm:K-irreps}
Every irreducible representation of $\Compact(H)$ is equivalent to the identity representation $\Compact(H)\to \Bounded(H)$.
\end{theorem}

\begin{proof}[Proof (outline)]
We use the fact that every bounded linear map $\Compact(H)\to \C$ has the form $k\mapsto \trace(tk)$ for some trace-class operator $t$. (See, eg, \cite[4.1.2]{Dixmier-Cstar-fr}\footnote{Readers of the english translation should beware that in the proof of Corollary 4.1.3, `operators with positive trace' should instead say `positive trace-class operators'}.) From this one deduces that the pure states on $\Compact(H)$ are precisely the maps $\phi_\xi:k\mapsto \langle \xi\, |\, k\xi\rangle$, where $\xi$ is a unit vector in $H$. The map $\Compact(H)\to H$, $k\mapsto k\xi$ then induces a unitary isomorphism from the GNS Hilbert space $H_{\phi_\xi}$ to $H$, that intertwines the GNS representation $\pi_{\phi_\xi}$ with the identity representation $\Compact(H)\to \Bounded(H)$.
\end{proof}

\subsection{Irreducible representations of $C_0(X,\Compact(H))$}\label{sec:CXKH}

Let $X$ be a locally compact Hausdorff space, and let $H$ be a nonzero Hilbert space.

\begin{theorem}\label{thm:CXK-irreps}
\begin{enumerate}[\rm(1)]
\item For each $x\in X$ the map 
\[
\ev_x : C_0(X,\Compact(H))\to \Compact(H),\qquad \ev_x(k)=k(x)
\]
is an irreducible representation.
\item Every irreducible representation of $C_0(X,\Compact(H))$ is equivalent to $\ev_x$ for a unique $x\in X$.
\end{enumerate}
\end{theorem}

\begin{proof}[Proof (outline)]
Each evaluation map $C_0(X,\Compact(H))\to \Compact(H)$ is surjective, and therefore an irreducible representation.

To show that every irreducible representation has this form, observe first that $C_0(X,\Compact(H))$ is a $C_0(X)$-module, in an obvious way. If $\pi:C_0(X,\Compact(H))\to \Bounded(H_\pi)$ is an irreducible representation then every element of $H_\pi$ can be written in the form $\pi(k)\xi$ for some $k\in C_0(X,\Compact(H))$ and $\xi\in H_\pi$ (Theorem \ref{thm:Cohen-Hewitt}), and the formula $\rho(f)\left( \pi(k)\xi\right) \coloneqq \pi(fk)\xi$  (for $f\in C_0(X)$, $k\in C_0(X,\Compact(H))$, and $\xi\in H_\pi$) defines a representation $\rho:C_0(X)\to \Bounded(H_\pi)$. Schur's lemma ensures that each $\rho(f)$ is a scalar operator, and therefore there is a unique $x\in X$ such that $\rho(f)=f(x)\id_{H_\pi}$ for each $f\in C_0(X)$. The representation $\pi$ is trivial on the ideal $\{ k\in C_0(X,\Compact(H))\ |\ k(x)=0\}$, and so $\pi$ factors through the quotient mapping
\[
\ev_x : C_0(X,\Compact(H)) \xrightarrow{k\mapsto k(x)} \Compact(H).
\]
Now Theorem \ref{thm:K-irreps} implies that $\pi$ is equivalent to $\ev_x$. 
\end{proof}

\subsection{Irreducible representations of direct sums}

If $A_i$ $(i\in I)$ is a family of $C^*$-algebras, then we define the direct sum
\[
\bigoplus_{i\in I} A_i = \{ (a_i)_{i\in I}\ |\ \|a_i\|\to 0 \text{ as }i\to\infty\}.
\]
This is a $C^*$-algebra, under coordinate-wise algebraic operations and with the norm $\|(a_i)_{i\in I}\|=\sup_{i\in I}\|a_i\|$.

\begin{exercise}\label{ex:direct-sum-irreps}
Prove that:
\begin{enumerate}[\rm(1)]
\item Each irreducible representation of $\bigoplus_{i\in I}A_i$ factors through the projection onto one of the summands $A_i$. 
\item Two irreducible representations of the direct sum are equivalent if and only if they factor through the same $A_i$, and are equivalent as representations of $A_i$.
\end{enumerate}
\end{exercise}

\subsection{Irreducible representations of $\Compact(H)^W$}

Let $\pi:W\to \Unitary(H)$ be a unitary representation of a finite group $W$. The formula
\(
\alpha_w(t)=\pi(w)t\pi(w)^*
\) 
defines an action of $W$ by $*$-automorphisms on the $C^*$-algebra $\Bounded(H)$, which restricts to an action on the  ideal $\Compact(H)$. We will compute the irreducible representations of the fixed-point algebra $\Compact(H)^W$, using the facts about representations of the finite group $W$ recalled in Appendix \ref{sec:compact-groups}.

We decompose $H$ into the direct sum $\bigoplus_{\rho\in \widehat{W}} H_{\pi,\rho}$ of its isotypical subspaces. Schur's lemma implies that there are no nonzero $W$-equivariant maps $H_{\pi,\rho}\to H_{\pi,\sigma}$ when $\rho\not\cong \sigma$, and so we have
\[
\Compact(H)^W = \bigoplus_{\rho\in \widehat{W}} \Compact(H_{\pi,\rho})^W.
\]
Exercise \ref{ex:direct-sum-irreps} gives an identification 
\[
\widehat{\Compact(H)^W} \cong \bigsqcup_{\rho \in \widehat{W}} \widehat{\Compact(H_{\pi,\rho})^W}.
\]

As in Appendix \ref{sec:compact-groups} (7) we consider, for each $\rho\in \widehat{W}$, the Hilbert space $\HS(\rho,\pi)^W$ of $W$-equivariant linear maps $H_\rho\to H$, with inner product $\langle r\, |\, s\rangle=\trace(r^*s)$. 

\begin{exercise}\label{ex:KW}
Prove:
\begin{enumerate}[\rm(1)]
\item For each $t\in \Bounded(H_{\pi,\rho})^W$ the formula $s\mapsto t\circ s$ defines a bounded linear map $\phi_{\pi,\rho}(t)\in \Bounded(\HS(\rho,\pi)^W)$. 
\item The map $\phi_{\pi,\rho}:\Bounded(H_{\pi,\rho})^W\to \Bounded(\HS(\rho,\pi)^W)$ is an isomorphism of $C^*$-algebras.
\item $\phi_{\pi,\rho}$ restricts to an isomorphism $\Compact(H_{\pi,\rho})^W \xrightarrow{\cong} \Compact(\HS(\rho,\pi)^W)$.
\end{enumerate}
\emph{Hint:} use the isomorphism $\mu_{\pi,\rho} : H_\rho \otimes \HS(\rho,\pi)^W \to H_{\pi,\rho}$ from Appendix \ref{sec:compact-groups}(7).
\end{exercise}

Since the $C^*$-algebra $\Compact(\HS(\rho,\pi)^W)$ has only one equivalence class of irreducible representations (Theorem \ref{thm:K-irreps}), we conclude:

\begin{theorem}\label{thm:KHW-irreps}
Let $\pi:W\to \Unitary(H)$ be a unitary representation of a finite group $W$, and consider the $C^*$-algebra 
\[
\Compact(H)^W = \{k\in \Compact(H)\ |\ \pi(w)k=k\pi(w)\text{ for all }w\in W\}.
\]
\begin{enumerate}[\rm(1)]
\item For each irreducible representation $\rho$ of $W$ occurring in $H$, the composition
\[
\Compact(H)^W \xrightarrow{k\mapsto \pi(e_\rho)k} \Compact(H_{\pi,\rho})^W \xrightarrow[\cong]{\phi_{\pi,\rho}} \Compact(\HS(\rho,\pi)^W)
\]
is an irreducible representation, which we shall also denote by $\phi_{\pi,\rho}$. (In other words, $\phi_{\pi,\rho}(k):s\mapsto k\circ s$, for $k\in \Compact(H)^W$ and $s\in \HS(\rho,\pi)^W$.)
\item Every irreducible representation of $\Compact(H)^W$ is equivalent to $\phi_{\pi,\rho}$ for some $\rho\subseteq\pi$.
\item $\phi_{\pi,\rho_1}$ and $\phi_{\pi,\rho_2}$ are equivalent as representations of $\Compact(H)^W$ if and only if $\rho_1$ and $\rho_2$ are equivalent as representations of $W$.\hfill\qed
\end{enumerate}
\end{theorem}

\subsection{Irreducible representations of $C_0(X,\Compact(H))^W$}\label{subsec:CXKW}

We will now consider a common generalisation of the $C^*$-algebras $C_0(X,\Compact(H))$ and $\Compact(H)^W$. To motivate the definition, think of $C_0(X,\Compact(H))$ as an algebra of endomorphisms of the (trivial)  bundle of Hilbert spaces $X\times H$ over $X$; and think of $\Compact(H)^W$ as an algebra of endomorphisms of the $W$-equivariant Hilbert space $H$ (which we can view as a bundle over a single point.) In this section we will combine these two scenarios by looking at $C^*$-algebras of endomorphisms of $W$-equivariant bundles of Hilbert spaces. Algebras of this form will turn out to be crucial to understanding $C^*_r(G)$ for real reductive groups $G$.

\begin{definition}\label{def:XH-action}
Let $X$ be a locally compact Hausdorff space, equipped with an action of a finite group $W$ by homeomorphisms, and let $H$ be a Hilbert space. An \emph{action of $W$ on $(X,H)$} is an action of $W$ on the product space $X\times H$ by homeomorphisms, such that
\begin{enumerate}[\rm(1)]
\item The projection $X\times H \xrightarrow{(x,\xi)\mapsto x} X$ is $W$-equivariant (\emph{i.e.}, $w(x,\xi)\mapsto wx$); and
\item For each $w\in W$ and each $x\in X$  the map $I_{w,x}:H\to H$ defined as the composition
\[
I_{w,x}:H\xrightarrow{\xi\mapsto(x,\xi)} \{x\}\times H \xrightarrow{\text{act by $w$}} \{wx\}\times H \xrightarrow{(wx,\xi)\mapsto\xi} H
\]
is a unitary linear map.
\end{enumerate}
\end{definition}

\begin{example}
If $X=\{x\}$ is a single point then an action of $W$ on $(X,H)$ is the same thing as a unitary representation $w\mapsto I_{w,x}$ of $W$ on $H$. 
\end{example}

\begin{exercise}\label{ex:XH-action}
\begin{enumerate}[\rm(1)]
\item Suppose that $W$ acts on $(X,H)$, as in Definition \ref{def:XH-action}. Prove that the operators $I_{w,x}:H\to H$ have the following properties:
\begin{enumerate}[\rm(i)]
\item For each $x\in X$ and all $w_1,w_2\in W$ we have $I_{w_1, w_2 x}\circ I_{w_2,x} = I_{w_1 w_2, x}$.
\item For each $w\in W$ the function $x\mapsto I_{w,x}$ is continuous in the strong operator topology.  (That is, for each $\xi\in H$ the function $X\to H$, $x\mapsto I_{w,x}\xi$, is continuous.)
\end{enumerate}
\item Let $X$ be a space equipped with an action of $W$, let $H$ be a Hilbert space, and suppose that $\{I_{w,x}:H\to H\ |\ w\in W,\ x\in X\}$ is a family of unitary operators satisfying properties (i) and (ii) above. Prove that the formula $w(x,\xi)\coloneqq (wx, I_{w,x}\xi)$ defines an action of $W$ on $(X,H)$ in the sense of Definition \ref{def:XH-action}.
\end{enumerate}

Thus an action of $W$ on $(X,H)$ is the same thing as an action of $W$ on $X$, together with a family of unitary operators $I = \{I_{w,x}\in \Unitary(H)\ |\ w\in W,\ x\in X\}$ having properties (i) and (ii).
\end{exercise}

\begin{exercise}\label{ex:beta-w-f-continuous}
Suppose that $W$ acts on $(X,H)$. For each $k\in C_0(X,\Compact(H))$ and each $w\in W$ we let $\alpha_w(k):X\to \Compact(H)$ be the function
\begin{equation}\label{eq:alpha-def}
\alpha_w(k)(x) = I_{w,w^{-1}x}k(w^{-1}x)I_{w^{-1},x}.
\end{equation}
The function $\alpha_w(k)$ is continuous for the strong-operator topology on $\Compact(H)$, because the family of operators $I_{w,x}$ is  strong-operator continuous in $x$ (Exercise \ref{ex:XH-action}). Prove that $\alpha_w(k)$ is, in fact, continuous for the \emph{norm} topology on $\Compact(H)$, and hence that the maps $\alpha_w$ define an action of $W$ on the $C^*$-algebra $C_0(X,\Compact(H))$.
\end{exercise}

\begin{exercise}\label{ex:fixed-point-crossed-product}
Let $X$ be a space with an action of a finite group $W$, and consider the Hilbert space $\ell^2 W$, with its standard orthonormal basis $\{\delta_v\ |\ v\in W\}$. For each $x\in X$ and $w\in W$ let $I_{w,x}\in \Unitary(\ell^2 W)$ be the left-translation operator $I_{w,x}=\lambda_w : \delta_v\mapsto \delta_{wv}$.  The criteria (i) and (ii) from Exercise \ref{ex:XH-action} are easily seen to hold for these operators, and so they define an action of $W$ on $(X,\ell^2 W)$, and hence an action $\alpha$ of $W$ on  $C_0(X,\Compact(\ell^2 W))$, as in Exercise \ref{ex:beta-w-f-continuous}. The action of $W$ on $X$ also gives rise to an action $\beta$ of $W$ on $C_0(X)$, defined by $\beta_w(f)(x)=f(w^{-1}x)$ for $f\in C_0(X)$ and $x\in X$, and hence to a crossed product $C^*$-algebra $C_0(X)\rtimes W$. For each $f\in C_0(X)$, $w\in W$,  and $x\in X$, let $\phi(fw)(x)$ be the operator on $\ell^2 W$ defined by  $\delta_v \mapsto f(wv^{-1}x)\delta_{vw^{-1}}$. Prove that this formula gives  an isomorphism of $C^*$-algebras $\phi: C_0(X)\rtimes W\xrightarrow{\cong} C_0(X,\Compact(\ell^2 W))^W$.
\end{exercise}

Returning to the general case, our goal in this section is to compute the irreducible representations of the fixed-point algebra $C_0(X,\Compact(H))^W$, for the action \eqref{eq:alpha-def} associated to an action of $W$ on $(X,H)$.

For each $x\in X$ let $W_x= \{w\in W\ |\ wx=x\}$. The map $I_{x}:  w\mapsto I_{w,x} $ is a unitary representation of $W_x$ on $H$. We write $H_x$ (instead of $H_{I_x}$) to denote the Hilbert space $H$ equipped with this representation, and write $\phi_{x,\rho}$ (instead of $\phi_{I_x,\rho}$) for the irreducible representation of $\Compact(H_x)^{W_x}$ associated to $\rho$  in Theorem \ref{thm:KHW-irreps}.

\begin{theorem}\label{thm:fixed-irreps}
Let $X$, $W$, $H$, and $I$ be as above.
\begin{enumerate}[\rm(1)]
\item For each $x\in X$ and each $\rho\in \widehat{W_x}$ that occurs in  $I_x$, the map 
\[
\pi_{x,\rho} = \phi_{x,\rho}\circ \ev_x:C_0(X,\Compact(H))^W \to \Bounded(\HS(\rho,I_x)^{W_x}),\quad \pi_{x,\rho}(k):s\mapsto k(x)\circ s
\]
is an irreducible representation. 
\item Every irreducible representation of $C_0(X,\Compact(H))^W$ is equivalent to $\pi_{x,\rho}$ for some $x$ and $\rho$.
\item Two irreducible representations $\pi_{x_1,\rho_1}$ and $\pi_{x_2,\rho_2}$ as in {\rm(1)}  are equivalent if and only if there is some $w\in W$ such that $x_2=wx_1$ and such that the representation $\rho_2$ of $W_{x_2}$ is equivalent to the representation $w\rho_1:v\mapsto \rho_1(w^{-1} v w)$.
\end{enumerate}
Consequently, the map $(Wx,\rho)\mapsto \pi_{x,\rho}$ is a bijection between the set 
\[
\{(Wx, \rho)\ |\ Wx\in W\backslash X,\ \rho\in \widehat{W_x},\ \rho\subseteq I_x\} 
\]
and $\left(C_0(X,\Compact(H))^W\right)^{\widehat{\ }}$.
\end{theorem}

\begin{proof}[Proof (outline)]
First note that the evaluation map $\ev_x:C_0(X,\Compact(H))^W\to \Compact(H)^{W_x}$ is surjective: to prove this, given $k\in \Compact(H)^{W_x}$, choose a function $h\in C_0(X)$ with $h(x)=1$ and $h(y)=0$ for all $y\in Wx\setminus \{x\}$. The function $hk:y\mapsto h(y)k$ lies in $C_0(X,\Compact(H))$, and the function
\begin{equation}\label{eq:f-average}
f\coloneqq \frac{1}{|W_x|}\sum_{w\in W} \alpha_w(hk)
\end{equation}
lies in $C_0(X,\Compact(H))^W$ and has $f(x)=k$. This surjectivity ensures that the composition of the evaluation map $\ev_x$ with the irreducible representation $\phi_{x,\rho}$ is an irreducible representation.

To show that every irreducible representation of $C_0(X,\Compact(H))^W$ is equivalent to one of the $\pi_{x,\rho}$s, recall that  Theorem \ref{thm:reps} says that every irreducible representation of $C_0(X,\Compact(H))^W$ is the restriction of an irreducible representation of $C_0(X,\Compact(H))$ to an irreducible, $C_0(X,\Compact(H))^W$-invariant subspace. The irreducible representations of $C_0(X,\Compact(H))$ are the evaluation maps $k\mapsto k(x)\in \Compact(H)$ (Theorem \ref{thm:CXK-irreps}). We noted above that the image of $C_0(X,\Compact(H))^W$ in such a representation is the subalgebra $\Compact(H)^{W_x}$ of $W_x$-equivariant compact operators, and so each irreducible $C_0(X,\Compact(H))^W$-invariant subspace of $H$ is equivalent to one of the irreducible representations $\phi_{x,\rho}$ identified in Theorem \ref{thm:KHW-irreps}.

Next take $x\in X$ and an irreducible representation $\rho$ of $W_x$ that occurs in $I_x$. Let $w\rho:W_{wx}\to \Unitary(H_\rho)$ be the irreducible representation $w\rho(v)=\rho(w^{-1}vw)$. The map
\[
\theta : \HS(\rho,I_x)^{W_x} \to \HS(w\rho,I_{wx})^{W_{wx}},\qquad s \mapsto I_{w,x}\circ s
\]
is a unitary isomorphism that intertwines the representations $\pi_{x,\rho}$ and $\pi_{wx,w\rho}$.

Finally, suppose that the representations $\pi_{x_1,\rho_1}$ and $\pi_{x_2,\rho_2}$ are equivalent. The algebra $C_0(X)^W$ of $W$-invariant $C_0$ functions on $X$ acts by pointwise multiplication on $C_0(X,\Compact(H))^W$. For each $h\in C_0(X)^W$, $k\in C_0(X,\Compact(H))^W$, and $x\in X$ we have $\ev_x(hk)=h(x)\ev_x(k)$, and consequently $\pi_{x,\rho}(hk)=h(x)\pi_{x,\rho}(k)$ for every $\rho\in \widehat{W_x}$. Since the representations $\pi_{x_1,\rho_1}$ and $\pi_{x_2,\rho_2}$ are equivalent they have equal kernels, which implies that every $h\in C_0(X)^W$ that vanishes at $x_1$ must also vanish at $x_2$. If $x_1$ and $x_2$ lay in distinct $W$-orbits then we could (by using Urysohn's lemma and then taking an average as in \eqref{eq:f-average}) construct a function $h\in C_0(X)^{W}$ that vanished at $x_1$ but not at $x_2$. So we must have $x_2=wx_1$ for some $w\in W$. The argument of the previous paragraph then shows that $\pi_{x_1,\rho_1}$---which we are assuming is equivalent to $\pi_{x_2,\rho_2}$---is also equivalent to $\pi_{x_2,w\rho_1}$. This implies that  $\phi_{x_2,\rho_2}$ and $\phi_{x_2,w\rho_1}$  are equivalent as representations of $\Compact(H_{x_2})^{W_{x_2}}$, and so Theorem \ref{thm:KHW-irreps} implies that $\rho_2$ and $w\rho_1$ are equivalent as representations of $W_{x_2}$.

The `consequently' part follows immediately from (1), (2), and (3).
\end{proof}

\begin{exercise}\label{ex:CXKW-example}
Let $X=\R$, and let $W$ be the group $\{1,w\}$ with two elements, acting on $\R$ by $wx=-x$. For each $x$ let
\[
I_{1,x} = \bmat{1 & 0 \\ 0 & 1} \quad \text{and}\quad I_{w,x} = \bmat{ e^{ix} & 0 \\ 0 & -e^{-ix}}\in \Unitary(\C^2).
\]
Check that this definition gives an action of $W$ on $(\R,\C^2)$; 
compute the dual $\left(C_0(\R,\Compact(\C^2))^W\right)^{\widehat{\ \ }}$ as a set; and prove that this dual is non-Hausdorff as a topological space.
\end{exercise}

\subsection{CCR $C^*$-algebras}

\begin{definition}
A $C^*$-algebra $A$ is said to be \emph{CCR} (for \emph{Completely Continuous Representations}), or \emph{liminal}, if for each irreducible representation $\pi:A\to \Bounded(H)$ we have $\pi(A)\subseteq \Compact(H)$ (equivalently, $\pi(A)=\Compact(H)$: see below.)
\end{definition}

\begin{examples}
The computations from the previous few sections show that $C_0(X)$, $\Compact(H)$, $C_0(X,\Compact(H))$, and $C_0(X,\Compact(H))^W$ are all CCR. Subalgebras (Theorem \ref{thm:reps}), quotients (Theorem \ref{thm:ideal-irreps}), and direct sums (Exercise \ref{ex:direct-sum-irreps}) of CCR $C^*$-algebras are CCR.
\end{examples}

\begin{theorem}[Harish-Chandra]\label{thm:G-CCR}
The reduced group $C^*$-algebra $C^*_r(G)$  of a real reductive group $G$ is CCR.
\end{theorem}

\begin{proof}[Proof (outline)]
Harish-Chandra proved that every irreducible unitary representation of $G$ is \emph{admissible}: that is to say, for each maximal compact subgroup $K$ of $G$, each irreducible representation $\rho$ of $K$, and each irreducible unitary representation $\pi$ of $G$, the isotypical subspace $H_{\pi,\rho}$ is finite-dimensional. 

Here is why this implies that $C^*_r(G)$ is CCR: for each finite subset $F\subset \widehat{K}$, let $e_F = \sum_{\rho\in F}e_\rho \in C(K)$, where $e_\rho$ is as in \eqref{eq:e-rho}. Admissibility implies that the projection $\pi(e_F)$ has finite rank, for each $\pi\in \widehat{G}_r$ and each finite $F\subset \widehat{K}$. Now, $C(K)$ acts on $C_c(G)$ by convolution over $K$. Exercise \ref{ex:isotypical-approximation} (below) implies that for each $f\in C_c(G)$, the operator $\pi(f)$ is the norm limit of the finite-rank operators $\pi(e_F\ast f)=\pi(e_F)\pi(f)$, and hence $\pi(f)$ is compact. Since $C_c(G)$ is dense in $C^*_r(G)$, this gives $\pi(C^*_r(G))\subseteq \Compact(H_\pi)$ as required.
\end{proof}

\begin{exercise}\label{ex:isotypical-approximation} 
For each $f\in C_c(G)$ we have $\|f-e_F\ast f\|_{C^*_r(G)}\to 0$ as $F$ runs through the directed set of finite subsets of $\widehat{K}$. (Hint: reduce it to a claim about $L^2$ convergence, and then use the decomposition $L^2(G)\cong \bigoplus_{\rho\in \widehat{K}} L^2(G)_\rho$, where $K$ acts on $L^2(G)$ by left translation.) 
\end{exercise}

\subsection{Kaplansky's Stone--Weierstrass theorem}

\begin{definition}\label{def:SWP}
A subalgebra $B$ of a $C^*$-algebra $A$ is called \emph{separating}\footnote{Dixmier uses the term \emph{riche}} if:
\begin{enumerate}[\rm(1), leftmargin=*]
\item every irreducible representation of $A$ remains irreducible upon restriction to $B$; and
\item inequivalent irreducible representations of $A$ remain inequivalent upon restriction to $B$.
\end{enumerate}
A $C^*$-algebra $A$ has the \emph{Stone--Weierstrass property} (SWP) if the only separating subalgebra of $A$ is $A$ itself.
\end{definition}

\begin{example}
Let $A=C_0(X)$. Recall from Theorem \ref{thm:CX-irreps} that every irreducible representation of $A$ is equivalent to the map $\ev_x:f\mapsto f(x)$ for precisely one $x\in X$. A subalgebra $B\subseteq A$ is separating if and only if:
\begin{enumerate}[\rm(1)]
\item for each $x\in X$ there is some $b\in B$ with $b(x)\neq 0$; and
\item for each pair of points $x\neq y$ in $X$ there is some $b\in B$ with $b(x)\neq b(y)$.
\end{enumerate}
The usual Stone--Weierstrass theorem thus says that $C_0(X)$ has the Stone--Weierstrass property for every $X$.
\end{example}

\begin{exercise}\label{ex:separating-irreps}
Suppose that $B$ is separating in $A$. Prove that the map $\pi\mapsto \pi\restrict_B$, sending each irreducible representation of $A$ to its restriction to $B$, induces a homeomorphism $\widehat{A}\xrightarrow{\cong}\widehat{B}$.
\end{exercise}

\begin{theorem}[Kaplansky]\label{thm:SW}
Every CCR $C^*$-algebra has the Stone--Weierstrass property.
\end{theorem}

We will give a reasonably detailed proof for the special case where $A=C_0(X,\Compact(H))$; then look at an example where $A=C_0(X,\Compact(H))^W$; and then outline the proof in the general case. See  \cite[11.1.6]{Dixmier-Cstar-fr} for a complete proof, and see  \cite{Akemann-Anderson} for related results and historical background. 

\paragraph*{Proof that $\Compact(H)$ has the Stone--Weierstrass property:} Since $\Compact(H)$ has only the one irreducible representation, a subalgebra $B\subseteq \Compact(H)$ is separating if and only if the inclusion $B\into \Bounded(H)$ is an irreducible representation. 

If this irreducibility does hold, then the {bicommutant theorem} (Theorem \ref{thm:SvN}) implies that $B$ is dense in $\Bounded(H)$---and therefore, in particular, in $\Compact(H)$---with respect to the strong-operator topology: that is, for every  $k\in \Compact(H)$ there is a net $(b_i)_{i\in I}$ of elements of $B$ with the property that $b_i \xi\to k\xi$ for every $\xi\in H$. 

Recall that every bounded linear map $\Compact(H)\to \C$ has the form $k\mapsto \trace(tk)$, for some trace-class operator $t$ on $H$. The following exercise thus implies that the only bounded linear map $\Compact(H)\to \C$ that vanishes on $B$ is the zero map; and then the Hahn--Banach theorem ensures that $B=\Compact(H)$. \hfill\qed

\begin{exercise}\label{ex:trace-k-continuous}
Let $t$ be a trace-class operator on $H$. Prove that the linear map $\Compact(H)\to \C$, $k\mapsto \trace(tk)$, is continuous in the strong-operator topology.
\end{exercise}

\paragraph*{Proof that $C_0(X,\Compact(H))$ has the Stone--Weierstrass property:} (after Elliott--Olesen \cite{Elliott-Olesen} and Takahasi \cite{Takahasi}.)

\bigskip

\noindent Suppose that $B$ is a separating $C^*$-subalgebra of $A=C_0(X,\Compact(H))$. 

\paragraph*{Claim 1:} The irreducible representations of $B$, up to equivalence, are precisely the evaluations $\ev_x:B\to \Compact(H)$ for $x\in X$.

\paragraph*{Proof of Claim 1:} See Exercise \ref{ex:separating-irreps}.

\paragraph*{Claim 2:} For each closed subset $Y\subseteq X$ let 
\[
J_Y = \{b\in B\ |\ b(y)=0\text{ for all }y\in Y\}.
\]
If $\ev_x(J_Y)= 0$ then $x\in Y$.

\paragraph*{Proof of Claim 2:} Suppose that $\ev_x(J_Y)=0$, so that the irreducible representation $\ev_x$ of $B$ factors through the quotient $B/J_Y$. Up to isomorphism we have $B/J_Y\subseteq C_0(Y,\Compact(H))$, and applying Theorem \ref{thm:reps} we conclude that the irreducible representation $\ev_x$ of $B/J_Y$ extends to an irreducible representation of $C_0(Y,\Compact(H))$. The irreducible representations of $C_0(Y,\Compact(H))$ are the $\ev_y$s for $y\in Y$, and so $x\in Y$.

\paragraph*{Claim 3:} If $I$ and $J$ are ideals in a $C^*$-algebra $C$ then $I+J$ is an ideal; and for each $c\in I+J$ and each $\epsilon>0$ we can write $c=c_I+c_J$, where $c_I\in I$, $c_J\in J$, $\|c_I\|<(1+\epsilon)\|c\|$, and $\|c_J\|<(2+\epsilon)\|c\|$.

\paragraph*{Proof of Claim 3:} The proof that $I+J$ is an ideal is mostly routine, except for the assertion that $I+J$ is closed in the norm; to prove this, note that the quotient mapping $A\to A/J$, being a homomorphism of $C^*$-algebras, is continuous and maps the $C^*$-algebra $I$ onto a closed subalgebra of $A/J$. 

For the second assertion: for each $c\in I+J$ we can of course write $c = c_I' + c_J'$ for some $c_I'\in I$ and $c_J'\in J$. The canonical isomorphism $(I+J)/J \cong I/(I\cap J)$ is an isomorphism of $C^*$-algebras, and is therefore isometric, which implies that
\[
\inf_{j\in J} \| c+j\| = \inf_{k\in I\cap J}\|c_I'+k\|.
\]
The left-hand side is at most $\|c\|$, so we can find $k\in I\cap J$ such that $\|c_I'+k\|< (1+\epsilon)\|c\|$. Set $c_I=c_I'+k\in I$, and $c_J=c-c_I=c_J'-k\in J$. We have $\|c_I\|< (1+\epsilon)\|c\|$ by construction, and $\|c_J\|\leq \|c\|+\|c_I\|<(2+\epsilon)\|c\|$.

\paragraph*{Claim 4:} $J_Y+J_Z=J_{Y\cap Z}$ and $J_Y\cap J_Z=J_{Y\cup Z}$ for all closed subsets $Y,Z\subseteq X$.

\paragraph*{Proof of Claim 4:} The equality $J_Y\cap J_Z=J_{Y\cup Z}$ follows immediately from the definitions. The inclusion $J_Y+J_Z\subseteq J_{Y\cap Z}$ is also obvious. For the reverse inclusion it will suffice (by Theorem \ref{thm:ideal-irreps}) to show that every irreducible representation of $B$ that annihilates the ideal $J_Y+J_Z$ also annihilates $J_{Y\cap Z}$. If $\ev_x$ annihilates $J_Y+J_Z$ then $\ev_x(J_Y)=\ev_x(J_Z)=0$, so Claim 2 ensures that $x\in Y\cap Z$, and consequently $\ev_x(J_{Y\cap Z})=0$, as required.

\paragraph*{Claim 5:} $B$ is a $C_b(X)$-submodule of $A$, where $C_b(X)$ denotes the algebra of bounded, continuous functions $X\to \C$. 

\paragraph*{Proof of Claim 5 \cite{Elliott-Olesen}:} Every element of $C_b(X)$ can be written as a linear combination of functions $f$ satisfying $f(X)\subseteq[0,1]$, so it will suffice to prove that $fb\in B$ for all such $f$ and all $b\in B$. Given such an $f$, a positive integer $n$, and $i\in \{0,1,\ldots,n\}$, we consider the closed subset
\[
Y_i = \left\{x\in X\ \middle|\ f(x)\not\in \left(\frac{i-1}{n},\frac{i+1}{n}\right) \right\}
\]
of $X$. Our assumption on $f$ ensures that $\bigcap_{i=0}^n Y_i=\emptyset$, and so Claim 4 implies that $J_{Y_0}+\cdots +J_{Y_n}=J_\emptyset=B$.

Now take $b\in B$. Repeated application of Claim 3 (with $\epsilon$ small enough that $(2+\epsilon)(1+\epsilon)^{n-2}<3$) shows that we can write $b=b_0+b_1 + \cdots + b_n$, where $b_i\in J_{Y_i}$ and $\|b_i\|<3\|b\|$ for each $i$. Define
\(
c_n = \sum_{i=0}^n \textstyle\frac{i}{n}b_i\in B.
\) 
For each $x\in X$ we have
\[
\begin{aligned}
\| (fb)(x)-c_n(x)\| & =  \left\| \sum_{i=0}^n (f(x) - \textstyle\frac{i}{n})b_i(x) \right\| \leq \sum_{i=0}^n \left|f(x)-\textstyle\frac{i}{n}\right| \|b_i(x)\|.
\end{aligned}
\]
If $x\in Y_i$ then $b_i(x)=0$, because $b_i\in J_{Y_i}$.  If $x\not\in Y_i$ then $f(x)\in \left(\frac{i-1}{n},\frac{i+1}{n}\right)$, and consequently $\left|f(x)-\frac{i}{n}\right|<\frac{1}{n}$. This second possibility arises for at most two $i$s, because each real number lies in at most two of the intervals $\left(\frac{i-1}{n},\frac{i+1}{n}\right)$. Recalling that $\|b_i\|<3\|b\|$, we continue the above computation to find
\[
\|(fb)(x)-c_n(x)\| \leq 2\times \textstyle\frac{1}{n}\times 3\|b\|.
\]
Thus the sequence $c_n\in B$ converges to $fb\in C_0(X,\Compact(H))$, and since $B$ is closed in the norm we conclude that $fb\in B$.\hfill\qed

\paragraph*{Claim 6:} $B=A$.

\paragraph*{Proof of Claim 6 \cite{Takahasi}:} Fix $a\in A$ and $\epsilon>0$. Let $Y=\{x\in X\ |\ \|a(x)\|\geq \epsilon\}$ (a compact subset of $X$), and let $U_0 = X\setminus Y$. 

For each $x\in X$ the map $\ev_x:B\to \Compact(H)$ is surjective (since $\ev_x$ is an irreducible representation of $B$, and $\Compact(H)$ has the SWP). So for each $x\in Y$ we can find $b_x\in B$ with $b_x(x)=a(x)$, and then by continuity we can find an open neighbourhood $U_x$ of $x$ such that $\|b_x(y)-a(y)\|<\epsilon$ for all $y\in U_x$. By compactness we can find a finite subset $\{x_1,\ldots,x_n\}$ of $Y$ such that $Y\subseteq \bigcup_{i=1}^n U_{x_i}$. To simplify the notation we set $b_i\coloneqq b_{x_i}$ and $U_i \coloneqq U_{x_i}$ for $i=1,\ldots,n$. The open sets $U_0,U_1,\ldots,U_n$ cover $X$, and we let $\{h_0,h_1,\ldots,h_n\}$ be a partition of $1$ subordinate to this cover.\footnote{This means that the $h_i$s are continuous functions $X\to [0,1]$; $h_i(x)=0$ for all $x\not\in U_i$; and $\sum_{i=0}^n h_i(x)=1$ for all $x\in X$. Finite open covers of  locally compact Hausdorff spaces always admit such partitions of $1$.} Define $b=\sum_{i=1}^n h_i b_i$. Claim 5 ensures that $b\in B$, and for each $x\in X$ we have
\begin{multline*}
 \| (a - b)(x) \|  = \left\| \left(ah_0 + \sum_{i=1}^n (a-b_i)h_i\right)(x)\right\| 
\\
\leq h_0(x)\| a(x) \| + \sum_{i=1}^n h_i(x)\|a(x)-b_i(x)\|  < \sum_{i=0}^n h_i(x) \epsilon = \epsilon.
\end{multline*}
Since $B$ is closed in $A$ we conclude that $B=A$. \hfill\qed

\paragraph*{Some exercises}

\begin{exercise}\label{ex:SW-extensions}
Suppose that $A$ is a $C^*$-algebra with an ideal $J$, such that  both $J$ and $A/J$ have the Stone--Weierstrass property. Prove that $A$ has the Stone--Weierstrass property. Deduce that if we have $A=J_0\supseteq J_1\supseteq \cdots \supseteq J_n=0$, where each $J_i$ is an ideal in $J_{i-1}$ and each of the quotients $J_{i-1}/J_i$ has the Stone--Weierstrass property, then $A$ has the Stone--Weierstrass property. 
\end{exercise}

\begin{exercise}\label{ex:SW-plane-example}
Consider the group of $2\times 2$ invertible matrices
\[
W = \langle s,t\rangle
\quad\text{where}\quad s = \bmat{1 & 0 \\ 0 & -1}\quad \text{and}\quad t = \bmat{0 & 1 \\ 1 & 0}.
\]
This group acts on the plane $\R^2$ in the obvious way. (It is the group of symmetries of the square with vertices $(\pm 1,0)$, $(0,\pm 1)$.)

For each $x\in \R^2$ and each $w\in W$ we define $I_{w,x}=w$, viewed now as a unitary operator on $\C^2$. Let $A=C_0(\R^2,\Mat_2)^W$, where $\Mat_2=\Compact(\C^2)$ is the $C^*$-algebra of $2\times 2$ matrices, and where $W$ acts on $C_0(\R^2,\Mat_2)$ as in \eqref{eq:alpha-def}. Find a finite composition series $A=J_0\supseteq J_1\supseteq\cdots \supseteq 0$ whose quotients $J_{i-1}/J_i$ have the Stone--Weierstrass property, and conclude that $A$ has the Stone--Weierstrass property.
\end{exercise}

\paragraph*{Outline of the proof of Stone--Weierstrass for an arbitrary CCR $A$:} Let $A$ be a CCR $C^*$-algebra. 
\begin{enumerate}[$\bullet$]
\item Show that if $A$ has a (possibly infinite) composition series whose quotients all have the Stone--Weierstrass property, then $A$ has the Stone--Weierstrass property. 
\item Show that $A$ has a composition series whose quotients are CCR with Hausdorff spectrum.
\item Every CCR $C^*$-algebra $A$ with Hausdorff spectrum $\widehat{A}$ is isomorphic to a $C^*$-algebra of $C_0$-sections of a \emph{continuous field} of $C^*$-algebras of compact operators.
\item The proof that we gave for $A=C_0(X,\Compact(H))$ adapts easily to general continuous fields.
\end{enumerate}
See  \cite[11.1.6]{Dixmier-Cstar-fr} for details.

\section{The reduced $C^*$-algebra of a real reductive group I}

We will now briefly explain how the the results of the previous section give a Fourier-type isomorphism between the reduced $C^*$-algebra of a real reductive group, and a certain space of compact-operator-valued functions on a parameter space of representations. See  \cite{Knapp-overview,WallachI,WallachII} for background on representations of real reductive groups, and see \cite{CCH} for details of this $C^*$-algebra computation.

Let $G$ be a real reductive group. Fix a Levi factor $L$ of a parabolic subgroup $P$ of $G$. (Example: $G=\GL_n(\R)$, $P$ is a subgroup of block-upper-triangular matrices, and $L$ is the corresponding subgroup of block-diagonal matrices.) The group $L$ has a Langlands decomposition $L=M_L\times A_L$, where $M_L$ has compact centre, and $A_L$ is isomorphic to its Lie algebra $\germ{a}_L$ via the exponential map. 

Let $\sigma$ be a \emph{discrete-series} representation of $M_L$: that is, $\sigma$ is an irreducible unitary representation of $M_L$ on a Hilbert space $H_\sigma$, for which the matrix coefficients $m\mapsto \langle \sigma(m)\xi\, |\, \eta\rangle$ are $L^2$ functions on $M_L$. Note that such representations might exist (\emph{e.g.}, if $M_L$ is compact, or if $M_L\cong \SL_2(\R)$), or they might not (\emph{e.g.}, if $M_L\cong \SL_n(\R)$ for $n\geq 3$.)

Each $\R$-linear functional $\chi\in \germ{a}_L^*$ gives a unitary character $\chi:e^x\mapsto e^{i\chi(x)}$ of $A_L$. Twisting $\sigma$ by these characters, we get a family of unitary representations $\{\sigma\otimes \chi\ |\ \chi\in \germ{a}_L^*\}$ of $L$.

The parabolic subgroup $P$ decomposes as a semidirect product $L\ltimes N$, where $N$ is the unipotent radical of $P$. The unitary representations $\sigma\otimes\chi$ of $L$ can be pulled back, along the quotient map $P\to L$, to give unitary representations of $P$. These representations then produce unitary representations $\Ind_P^G(\sigma\otimes\chi)$ of $G$, via unitary induction. This procedure for turning representations of $L$ into representations of $G$ is called \emph{parabolic induction}.

The parabolically induced representation $\Ind_P^G(\sigma\otimes\chi)$ is initially defined on a Hilbert space of $P$-equivariant functions $G\to H_\sigma$; the $P$-action on $H_\sigma$ depends on $\chi$, so the Hilbert space does likewise. But the Iwasawa decomposition $G=KP$, where $K$ is a suitable maximal compact subgroup of $G$, implies that restriction of functions from $G$ to $K$ gives a bijection between $P$-equivariant functions on $G$, and $P\cap K$-equivariant functions on $K$. Since $P\cap K$ is contained in $M_L$, on which the representations $\sigma\otimes\chi$ all agree, we conclude that the representations $\Ind_P^G(\sigma\otimes \chi)$ can all be realised on the same Hilbert space. We shall denote this space by $\Ind_P^G H_\sigma$. (This is the \emph{compact picture} of parabolic induction.)

Each of the representations $\Ind_P^G(\sigma\otimes\chi)$ is {tempered}, and thus induces a $C^*$-algebra representation $C^*_r(G)\to \Bounded(\Ind_P^G H_\sigma)$ (\emph{cf}.~Theorem \ref{thm:tempered}.) The fact that $\sigma$ is a square-integrable representation of $M_L$, and that $P$ is co-compact in $G$, ensures that $C^*_r(G)$  acts on $\Ind_P^G H_\sigma$ by compact operators. Moreover, a version of the Riemann--Lebesgue lemma ensures that for each $f\in C^*_r(G)$ the function $\chi\mapsto \Ind_P^G(\sigma\otimes\chi)(f)$ is a $C_0$ function, from $\germ{a}_L^*$ to $\Compact(\Ind_P^G H_\sigma)$. We thus have a homomorphism of $C^*$-algebras
\[
\pi_{(L,\sigma)} : C^*_r(G)\to C_0(\germ{a}_L^*,\Compact(\Ind_P^G H_\sigma)).
\]

This map $\pi_{(L,\sigma)}$ is, in general, neither injective nor surjective. The failure of injectivity can be remedied by considering all possible pairs $(L,\sigma)$ at once, as we shall see below. The failure of surjectivity of $\pi_{(L,\sigma)}$ is linked to the existence of intertwining operators between the representations $\Ind_P^G(\sigma\otimes \chi)$. The  Weyl group $\operatorname{Norm}_G(A_L)/L$ acts on the unitary duals of both $M_L$ and $A_L$ by conjugation, and $W_\sigma$ denotes the stabiliser in this Weyl group of the isomorphism class of $\sigma$. Knapp and Stein's theory of normalised intertwining operators gives, for each $w\in W_\sigma$ and each $\chi\in \germ{a}_L^*$, a unitary operator $I_{w,\chi}$ on $\Ind_P^G H_\sigma$ satisfying 
\begin{equation}\label{eq:I-G-intertwining}
I_{w,\chi} \Ind_P^G(\sigma\otimes \chi)(g) = \Ind_P^G(\sigma\otimes w\chi)(g)I_{w,\chi}
\end{equation} 
for all $g\in G$. For each $w\in W_\sigma$ the map $\chi\mapsto I_{w,\chi}$ is strong-operator continuous, and Knapp has shown in \cite{Knapp-Commutativity} that we can arrange things so that $I_{w_2,w_1\chi}I_{w_1,\chi}=I_{w_2 w_1,\chi}$ for all $w_1,w_2\in W_\sigma$ and all $\chi\in \germ{a}_L^*$. That is to say, the operators $I_{w,\chi}$ define an {action} of $W_\sigma$ on the pair $(\germ{a}_L^*,\Ind_P^G H_\sigma)$, in the sense of Definition \ref{def:XH-action} (\emph{cf.}~Exercise \ref{ex:XH-action}.) This action induces an action of $W_\sigma$ on the $C^*$-algebra $C_0(\germ{a}_L^*,\Compact(\Ind_P^G H_\sigma))$ as in \eqref{eq:alpha-def}, and we let $C_0(\germ{a}_L^*,\Compact(\Ind_P^G H_\sigma))^{W_\sigma}$ denote the fixed-point subalgebra for this action. The intertwining property \eqref{eq:I-G-intertwining} of the operators $I_{w,\chi}$ ensures that the homomorphism $\pi_{(L,\sigma)}$ maps $C^*_r(G)$ into this fixed-point algebra.

\begin{theorem}\label{thm:Plancherel}
Let $G$ be a real reductive group. The maps $\pi_{(L,\sigma)}$ assemble to give an isomorphism of $C^*$-algebras
\[
\bigoplus_{[L,\sigma]} {\pi_{(L,\sigma)}}: C^*_r(G) \xrightarrow{\cong} \bigoplus_{[L,\sigma]} C_0(\germ{a}_L^*, \Compact(\Ind_P^G H_\sigma))^{W_\sigma},
\]
where the direct sum is over the set of $G$-conjugacy classes of pairs $(L,\sigma)$ as above.
\end{theorem}

\begin{proof}[Proof (outline)]
We must show that the homomorphism $\bigoplus \pi_{(L,\sigma)}$ has the following properties:
\begin{enumerate}[\rm(1)]
\item the homomorphism maps into the direct \emph{sum}: \emph{i.e.}, for each $a\in C^*_r(G)$ and each $\epsilon>0$ there are, up to conjugacy, only finitely many $(L,\sigma)$s for which $\sup_{\chi \in \germ{a}_L^*}\|\Ind_P^G(\sigma\otimes \chi)(a)\|\geq \epsilon$;
\item the homomorphism is injective; that is, for each $a\in C^*_r(G)$ there is some $(L,\sigma)$ and some $\chi\in \germ{a}_L^*$ with $\Ind_P^G(\sigma\otimes \chi)(a)\neq 0$; and
\item the homomorphism maps surjectively onto the direct sum.
\end{enumerate}
Each of these three assertions relies on a major theorem in representation theory. 

The assertion (1) follows from a theorem of Harish-Chandra \cite[Lemma 70]{HC-DSII}, saying that each irreducible representation of a maximal compact subgroup $K$ of $G$ occurs in $\Ind_P^G(\sigma\otimes \chi)$ for only finitely many $[L,\sigma]$s. This implies that for each $\rho\in \widehat{K}$ we have $\pi_{(L,\sigma)}(e_\rho C^*_r(G))=0$ for all but finitely many $[L,\sigma]$; here $e_\rho$ is the $\rho$-isotypical projection, as in \eqref{eq:e-rho}. Since the span of the subspaces $e_{\rho}C^*_r(G)$  is dense in $C^*_r(G)$ (\emph{cf}.~Exercise \ref{ex:isotypical-approximation}), this proves assertion (1). 

Assertion (2) follows from a theorem of Langlands \cite[Lemma 4.10]{Langlands} and Trombi \cite[Corollary 6.2]{Trombi}, which says that each tempered irreducible unitary representation of $G$---and therefore, each irreducible representation of $C^*_r(G)$---occurs as a direct-summand in some $\Ind_P^G(\sigma\otimes \chi)$. 

Finally, assertion (3) follows from another theorem of Harish-Chandra \cite[38 Theorem 1]{HC-HAIII}, which says that the operators $I_{w,\chi}$ (and linear combinations thereof) are the only intertwiners between the representations $\Ind_P^G(\sigma\otimes \chi)$. On the other hand, Theorem \ref{thm:fixed-irreps} implies that the $I_{w,\chi}$s are also intertwiners between the irreducible representations of $\bigoplus_{[L,\sigma]} C_0(\germ{a}_L^*,\Compact(\Ind_P^G H_\sigma))^{W_\sigma}$, and so the Stone--Weierstrass theorem ensures that the image of $C^*_r(G)$ under the homomorphism $\bigoplus_{[L,\sigma]}\pi_{(L,\sigma)}$ is equal to the direct sum.
\end{proof}

\section{Hilbert modules and Morita equivalence}\label{sec:Morita}

In many applications of $C^*$-algebra theory, including the representation-theoretic computation that is our focus here, the $C^*$-algebras themselves are less important than certain invariants of these $C^*$-algebras, such as their spectra and $K$-theory. Morita equivalence gives us a way of replacing one $C^*$-algebra by another while keeping these invariants unchanged.

In algebra Morita equivalence is usually defined as an equivalence of module categories, and one proves that every such equivalence is given by tensor product with a certain kind of bimodule. In $C^*$-algebra theory the definition in terms of bimodules came first, and this is the course that we shall follow. Much of this material is due to Rieffel; see \cite{Rieffel-survey} for an overview and more detailed references, and see \cite{Raeburn-Williams} or \cite{Lance} for a detailed exposition.

\subsection{Hilbert modules}

\begin{definition}\label{def:Hilbert-module}
A   \emph{right Hilbert module} over a $C^*$-algebra $B$ is a $\C$-vector space $E$ that is a right $B$-module, and that is equipped with a $B$-valued inner product $\langle\ |\ \rangle : E\times E\to B$ satisfying:
\begin{enumerate}[\rm(1)]
\item $\langle \xi\,|\,\eta \rangle$ is linear in $\eta$ and conjugate-linear in $\xi$.
\item $\langle \xi b_1\, |\, \eta b_2\rangle = b_1^*\langle \xi\, |\, \eta\rangle b_2$ for all $\xi,\eta\in E$ and $b_1,b_2\in B$.
\item $\langle \eta\, |\, \xi\rangle = \langle \xi\, |\, \eta\rangle^*$ for all $\xi,\eta\in E$.
\item $\langle \xi\, |\, \xi\rangle \geq 0$ for all $\xi\in E$, and $\langle \xi\, |\, \xi\rangle =0$ only if $\xi=0$.
\item $E$ is complete in the norm $\|\xi\|_E=\|\langle \xi\, |\, \xi\rangle\|_B^{1/2}$.
\end{enumerate}
(The notion of a \emph{left} Hilbert module  is defined in a similar way. In these notes `Hilbert module' will mean `right Hilbert module' unless otherwise indicated.)
\end{definition}

\begin{remark}
The conditions (1)--(4) in Definition \ref{def:Hilbert-module} imply that the function $\|\ \|_E$ in (5) is in fact a norm on $E$. A standard way to prove this involves first proving the inequality $\langle \eta\, |\, \xi\rangle \langle \xi, |\, \eta\rangle \leq \| \xi\|^2_E \langle \eta\, |\, \eta\rangle$ in the $C^*$-algebra $B$ (see,  \emph{e.g.},  \cite[Lemma 2.5]{Raeburn-Williams}.) Taking norms in this $C^*$-algebra-valued Cauchy--Schwarz inequality gives a Cauchy--Schwarz inequality $\| \langle \xi\, |\, \eta\rangle\| \leq \|\xi\| \|\eta\|$, which implies that $\|\ \|_E$ satisfies the triangle inequality. \end{remark}

\begin{example}\label{example:B-over-B}
Each $C^*$-algebra $B$ is a Hilbert module over itself, with inner product $\langle b_1\,|\,b_2\rangle = b_1^*b_2$. The Hilbert-module norm is equal to the $C^*$-algebra norm. 
\end{example}

\begin{example}
Let $H$ be a Hilbert space (regarded as a {right} module over $\C$). The $\C$-valued inner product on $H$ makes $H$ into a  Hilbert $\C$-module.
\end{example}

\begin{example}\label{example:CXH}
Let $H$ be a Hilbert space, and let $X$ be a locally compact Hausdorff space. Then $C_0(X,H)$ is a Hilbert module over the $C^*$-algebra $C_0(X)$: the module structure is defined by pointwise multiplication, and the $C_0(X)$-valued inner product is $\langle \xi\, |\, \eta\rangle(x)\coloneqq \langle \xi(x)\, |\, \eta(x)\rangle$ (for $\xi,\eta\in C_0(X,H)$ and $x\in X$.)
\end{example}

The following fact is sometimes useful when working with Hilbert modules:

\begin{lemma}\label{lem:E=EB}
If $E$ is a  Hilbert $B$-module, then every element of $E$ can be written in the form $\xi b$ where $\xi\in E$ and $b\in B$. 
\end{lemma}

\begin{proof}[Proof (outline)]
This follows from Theorem \ref{thm:Cohen-Hewitt}, combined with the assertion that the elements $\xi b$ have dense span in $E$. For a proof of the latter assertion see \cite[Proposition 2.31]{Raeburn-Williams} (which in fact proves that each $\xi\in E$ can be written as $\eta\langle\eta\, |\, \eta\rangle$ for a unique $\eta\in E$.)
\end{proof}

\begin{exercise}\label{ex:full}
Let $E$ be a right Hilbert $B$-module. 
\begin{enumerate}[\rm(1)]
\item Prove that the subset $J=\overline{\lspan}\{\langle \xi\,|\,\eta\rangle\in B\ |\ \xi,\eta\in E\}$ is an ideal in $B$. 
\end{enumerate}
We say that $E$ is \emph{full} over $B$ if this ideal $J$ is equal to $B$. Note that $E$ is tautologically a full right Hilbert module over $J$.
\begin{enumerate}
\item[\rm(2)] Prove that $C_0(X,H)$ is full as a right Hilbert $C_0(X)$-module, if $H\neq 0$.
\end{enumerate}
\end{exercise}

\begin{exercise}\label{ex:V-tensor-E}
\begin{enumerate}[\rm(1)]
\item Let $E$ and $F$ be Hilbert $B$-modules. On the $B$-module direct sum $E\oplus F$ we define a $B$-valued inner product by $\langle (\xi_1,\eta_1)\, |\, (\xi_2,\eta_2)\rangle \coloneqq \langle \xi_1\, |\, \xi_2\rangle + \langle \eta_1\, |\, \eta_2\rangle$. Prove that this inner product makes $E\oplus F$ into a Hilbert $B$-module.
\item Let $V$ be a finite-dimensional Hilbert space, and let $E$ be a Hilbert module over a $C^*$-algebra $B$. On the  tensor product $V\otimes_{\C} E$ we define $(v\otimes \xi)b\coloneqq v\otimes \xi b$ and $\langle v_1\otimes\xi_1\, |\, v_2\otimes \xi_2\rangle \coloneqq \langle v_1\, |\, v_2\rangle \langle \xi_1\, |\, \xi_2\rangle\in B$. Prove that these definitions make $V\otimes_{\C} E$ into a Hilbert $B$-module.
\end{enumerate}
\end{exercise}

\subsection{Operators on Hilbert modules}

\begin{definition}
Let $E$ and $F$ be Hilbert modules over a $C^*$-algebra $B$.
\begin{enumerate}[\rm(1)]
\item A linear map $t:E\to F$ is said to be \emph{adjointable} if there is a linear map $t^*:F\to E$ satisfying $\langle \eta\, |\, t\xi \rangle = \langle t^*\eta\, |\, \xi \rangle$ for all $\xi\in E$ and $\eta\in F$. The set of all such maps is denoted $\Adjointable_B(E,F)$.
\item The \emph{$B$-compact} (or just \emph{compact}) linear maps $E\to F$ are those in the operator-norm closure of the linear span of the linear maps $\ket{\eta}\bra{\xi}:E\to F$, $\zeta\mapsto \eta\langle\xi\ |\ \zeta\rangle$, for $\eta\in F$ and $\xi,\zeta\in E$. The set of all such maps is denoted $\Compact_B(E,F)$.
\end{enumerate}
\end{definition}

Unlike the case of Hilbert spaces, bounded linear maps between Hilbert modules are not necessarily adjointable:

\begin{exercise}\label{ex:non-adjointable}
Regard both the $C^*$-algebra $C([0,1])$ and its ideal $C_0( (0,1])$ as Hilbert modules over $C([0,1])$, as in Example \ref{example:B-over-B}. Prove that the inclusion map $C_0((0,1])\into C([0,1])$ is not adjointable.
\end{exercise}

However, arguments similar to those in the case of $B=\C$ do show:

\begin{proposition}
\begin{enumerate}[\rm(1)]
\item Every adjointable map of Hilbert $B$-modules is bounded and $B$-linear; and every $B$-compact linear map is adjointable.
\item $\Adjointable_B(E,F)$ is an operator-norm-closed subspace of the Banach space of all bounded linear maps from $E$ to $F$; and  $\Adjointable_B(E)\coloneqq \Adjointable_B(E,E)$ is a $C^*$-algebra.
\item $\Compact_B(E,F)$ is an operator-norm-closed subspace of $\Adjointable_B(E,F)$, and $\Compact_B(E)\coloneqq \Compact_B(E,E)$ is an ideal in the $C^*$-algebra $\Adjointable_B(E)$.\hfill\qed
\end{enumerate}
\end{proposition}

\begin{exercise}\label{ex:K-C_0-H}
Let $H$ be a Hilbert space, and consider the Hilbert $C_0(X)$-module $C_0(X,H)$ (Example \ref{example:CXH}.) Prove that $\Compact_{C_0(X)}(C_0(X,H)) \cong C_0(X,\Compact(H))$.
\end{exercise}

\begin{exercise}\label{exercise:E-dual}
Let $E$ be a Hilbert $B$-module. Regard $B$ as a Hilbert $B$-module as in Example \ref{example:B-over-B}.
\begin{enumerate}[\rm(1)]
\item Prove that $\Compact_B(E,B)=\{ \bra{\xi}:\eta\mapsto \langle \xi\, |\, \eta\rangle\ |\ \xi\in E\}$.
\item Deduce  that the map sending each $b\in B$ to the multiplication operator $m_b:B\xrightarrow{c\mapsto bc} B$ gives an isomorphism of $C^*$-algebras $B\xrightarrow{\cong} \Compact_B(B)$. 
\item Prove that the formula $\llangle k\, |\, l\rrangle\coloneqq k^* l$, for $k,l\in \Compact_B(E,B)$, makes $\Compact_B(E,B)$ into a full Hilbert module over the $C^*$-algebra $\Compact_B(E)$. \emph{Hint:} consider the $C^*$-algebra $\Compact_B(E\oplus B)$, where $E\oplus B$ is a Hilbert $B$-module as in Exercise \ref{ex:V-tensor-E}. In matrix form  we can write 
\[
\Compact_B(E\oplus B) = \begin{bmatrix} \Compact_B(E) & \Compact_B(E,B) \\ \Compact_B(B,E) & \Compact_B(B)\end{bmatrix},
\]
where in each matrix entry, the indicated space of compact operators is embedded isometrically.
\item Now suppose that $E$ is full over $B$. Prove that the map sending $b\in B$ to the map $\Compact_B(E,B)\xrightarrow{k\mapsto m_b\circ k} \Compact_B(E,B)$ is an isomorphism of $C^*$-algebras $B\xrightarrow{\cong} \Compact_{\Compact_B(E)}( \Compact_B(E,B))$.
\item Now consider $\Compact_B(B,E)$. Show that the inner product $\langle k\, |\, l\rangle \coloneqq k^*l$ makes $\Compact_B(B,E)$ into a right Hilbert module over $\Compact_B(B)\cong B$, and that the map sending $\xi\in E$ to the operator $\ket{\xi}:b\mapsto \xi b$ is an inner-product-preserving isomorphism of Hilbert $B$-modules $E\xrightarrow{\cong} \Compact_B(B,E)$.
\end{enumerate}
\end{exercise}

A consideration of operators on Hilbert modules gives a convenient resolution of a point that we left pending in Example \ref{examples-Cstar}(8.ii): the definition of the norm on a crossed-product $C^*$-algebra.

\begin{exercise}\label{ex:crossed-product-norm}
Let $B$ be a $C^*$-algebra equipped with an action $\beta$ of a finite group $W$, by $*$-automorphisms. Consider the tensor product $\ell^2 W\otimes_{\C} B$, which is a Hilbert $B$-module as in Exercise \ref{ex:V-tensor-E}. For each $b\in B$ and $w\in W$,  let $t_{bw}$ be the operator on $\ell^2 W\otimes_{\C} B$ defined by $t_{bw}(\delta_v\otimes a) = \delta_{wv}\otimes \beta_{v^{-1}w^{-1}}(b)a$. (Here $\{\delta_v\ |\ v\in W\}$ denotes the standard orthonormal basis for $\ell^2 W$.) Show that the map $bw\mapsto t_{bw}$ gives an injective homomorphism of $*$-algebras $B\rtimes W\to \Adjointable_B(\ell^2 W\otimes_{\C} B)$, and hence that $\|\sum_w b_w w\|\coloneqq \|\sum_w t_{b_w w}\|_{\Adjointable_B(\ell^2 W\otimes B)}$ is a $C^*$-algebra norm on $B\rtimes W$.
\end{exercise}

\subsection{Morita equivalence}

\begin{definition}\label{def:Morita-bimodule}
Let $A$ and $B$ be $C^*$-algebras. A \emph{Morita equivalence} from $A$ to $B$ is a full Hilbert $B$-module $E$, equipped with an isomorphism of $C^*$-algebras $A\xrightarrow{\cong} \Compact_B(E)$. We say that $C^*$-algebras $A$ and $B$ are \emph{Morita equivalent} (notation: $A\sim B$) if there is a Morita equivalence from $A$ to $B$.\end{definition}

\begin{remark}
Here is an equivalent formulation of the definition: a Morita equivalence from $A$ to $B$ is an $A$-$B$-bimodule $E$ 
that is a full left Hilbert $A$-module (with inner product $[\ |\ ]$) and a full right Hilbert $B$-module (with inner product $\langle\ |\ \rangle$), satisfying $[\xi b\,|\,\eta] = [\xi\,|\,\eta b^*]$, $\langle a\xi\,|\,\eta\rangle = \langle \xi\,|\,a^*\eta\rangle$, and $[\xi\,|\,\eta]\zeta = \xi\langle \eta\,|\,\zeta\rangle$ for all $\xi,\eta,\zeta\in E$, all $a\in A$, and all $b\in B$. The fact that this definition is equivalent to Definition \ref{def:Morita-bimodule} can be established using arguments similar  to those in Exercise \ref{exercise:E-dual}.
\end{remark}

\begin{example}
Let $H$ be a nonzero Hilbert space. 
\begin{enumerate}[\rm(1)]
\item $H$ is a Morita equivalence from $\Compact(H)$ to $\C$.
\item For each locally compact Hausdorff space $X$, $C_0(X,H)$ is a Morita equivalence from $C_0(X,\Compact(H))$  to $C_0(X)$. 
\end{enumerate}
\end{example}

Here are some important facts about Morita equivalence:

\begin{theorem}
\begin{enumerate}[\rm(1)]
\item Morita equivalence is an equivalence relation. {\rm\cite{Rieffel-induced}}
\item $A\sim B$ if and only if $A\otimes \Compact(H)\cong B\otimes \Compact(H)$, so long as $A$ and $B$ have \emph{countable approximate identities} (which is the case for all of the $C^*$-algebras of interest in these notes.) {\rm\cite{BGR}}
\item If $A\sim B$ then $\widehat{A}$ is homeomorphic to $\widehat{B}$, and $A$ and $B$ have isomorphic $K$-theory. {\rm\cite{Rieffel-induced,BGR,Kasparov-equivariant-KK}}
\item $A\sim B$ if and only if $A$ and $B$ have equivalent categories of \emph{operator modules}. {\rm\cite{Blecher-Morita}}
\item If $A$ and $B$ both have $1$, then $A\sim B$ if and only if $A$ and $B$ are Morita equivalent as rings (ie, their categories of purely algebraic modules are equivalent.) {\rm\cite{Beer}}
\end{enumerate}
\hfill\qed
\end{theorem}

\begin{exercise}\label{ex:Morita-symmetric}
Use Exercise \ref{exercise:E-dual} to prove that Morita equivalence is symmetric: if $A\sim B$ then $B\sim A$.
\end{exercise}

\begin{exercise}\label{ex:Morita-sum}
Suppose that $A_i$ and $B_i$ (indexed by $i\in I$) are two families of $C^*$-algebras, such that $A_i\sim B_i$ for each $i$. Prove that $\bigoplus_i A_i \sim \bigoplus_i B_i$.
\end{exercise}

\subsection{Equivariant Hilbert modules}

Suppose that a finite group $W$ acts on a $C^*$-algebra $B$ by $*$-automorphisms $\beta_w$.

\begin{definition}\label{def:equivariant-Hilbert-module}
A \emph{$W$-equivariant} Hilbert $B$-module is a Hilbert $B$-module $E$ equipped with a family of invertible linear operators $\gamma_w:E\to E$ satisfying $\gamma_w(\xi b) = \gamma_w(\xi)\beta_w(b)$, $\langle \gamma_w \xi\, |\, \gamma_w\eta\rangle = \beta_w(\langle \xi\, |\, \eta\rangle)$, and $\gamma_{w_1}\circ \gamma_{w_2}=\gamma_{w_1w_2}$ for all $w,w_1,w_2\in W$, $\xi,\eta\in E$, and $b\in B$.
\end{definition}

\begin{exercise}\label{ex:W-acts-on-K}
Let $E$ be a $W$-equivariant Hilbert $B$-module. Prove that the formula $\alpha_w(k) \coloneqq \gamma_w k \gamma_w^{-1}$ (for $w\in W$ and $k\in \Compact_B(E)$) defines an action of $W$ on the $C^*$-algebra $\Compact_B(E)$.
\end{exercise}

\begin{example}\label{example:CXH-equivariant}
Suppose that $W$ acts on a pair $(X,H)$ as in Definition \ref{def:XH-action} (\emph{cf.}~Exercise \ref{ex:XH-action}): that is, suppose that $W$ acts on the space $X$ by homeomorphisms, and that we have a strong-operator-continuous family of unitary operators $I_{w,x}$ on the Hilbert space $H$ satisfying $I_{w_1,w_2x}I_{w_2,x}=I_{w_1w_2,x}$ for all $w_1,w_2\in W$ and all $x\in X$. Let $\beta$ denote the induced action of $W$ on $B=C_0(X)$ (\emph{i.e.}, $\beta_w f(x) = f(w^{-1}x)$), and for each $w\in W$ let $\gamma_w$ be the linear map on the Hilbert $C_0(X)$-module $C_0(X,H)$ defined by $\gamma_w\xi (x) \coloneqq I_{w,w^{-1}x}(\xi(w^{-1}x))$. {\bf Exercise:} check that these operators make $C_0(X,H)$ into a $W$-equivariant Hilbert $C_0(X)$-module, and that the $W$-action on $C_0(X,\Compact(H))\cong \Compact_{C_0(X)}(C_0(X,H))$ defined in Exercise \ref{ex:W-acts-on-K} is the same as the one defined in \eqref{eq:alpha-def}.
\end{example}

There is a crossed-product construction for equivariant Hilbert modules, that is compatible with the crossed-product construction for $C^*$-algebras (\emph{cf.}~Examples \ref{examples-Cstar}(8.ii)): 

\begin{definition}\label{def:E-cross-W}
Let $E$ be a $W$-equivariant Hilbert $B$-module. We make the vector space
\(
E\rtimes W \coloneqq \left\{ \sum_{w\in W} \xi_w w\ \middle|\ \xi_w\in E \right\}
\)
into a right $B\rtimes W$-module by defining $(\xi w_1)(b w_2) \coloneqq \xi \beta_{w_1}(b) w_1 w_2$,  and we equip $E\rtimes W$ with the $B\rtimes W$-valued inner product defined by
\[
\langle \xi_1 w_1\, |\, \xi_2 w_2\rangle \coloneqq w_1^{-1}\langle \xi_1\, |\, \xi_2\rangle w_2 = \beta_{w_1^{-1}}(\langle \xi_1\, |\, \xi_2\rangle ) w_1^{-1} w_2.
\] 
\end{definition}

\begin{theorem}\label{thm:E-cross-W}
These definitions make $E\rtimes W$ into a Hilbert $B\rtimes W$-module, which is full  over $B\rtimes W$ if $E$ is full over $B$. We have $\Compact_{B\rtimes W}(E\rtimes W) \cong \Compact_B(E)\rtimes W$, where $W$ acts on $\Compact_B(E)$ as in Exercise \ref{ex:W-acts-on-K}.
\end{theorem}

\begin{proof}[Proof (outline)]
The direct sum $E\oplus B$ is a Hilbert $B$-module, as in Exercise \ref{ex:V-tensor-E}. It is, in a natural way, a $W$-equivariant Hilbert $B$-module: each $w\in W$ acts on $E\oplus B$ by the operator $(\xi,b)\mapsto (\gamma_w(\xi), \beta_w(b))$. This $W$-action on $E\oplus B$ induces a $W$-action on the $C^*$-algebra $\Compact_B(E\oplus B)$, as in Exercise \ref{ex:W-acts-on-K}. Use the identification $E\cong \Compact_B(B,E)$ (Exercise \ref{exercise:E-dual}) to embed $E\rtimes W$ as a closed subspace of $\Compact_B(E\oplus B)\rtimes W$, and then note that the inner product defined in Definition \ref{def:E-cross-W} is the restriction to $E\rtimes W$ of the standard inner product $\langle k\, |\, l\rangle = k^*l$ on $\Compact_B(E\oplus B)\rtimes W$.

The assertion about fullness follows easily from the observation that for all $\xi,\eta\in E$ and $w\in W$ we have $\langle \xi 1_W\, |\, \eta w\rangle = \langle \xi\, |\, \eta\rangle w$.

For the assertion about compact operators: Given $k\in \Compact_B(E)$ and $w\in W$ define an operator $\phi(kw)$ on $E\rtimes W$ by $\phi(kw):\xi v \mapsto k(\gamma_w\xi) wv$ (for $\xi\in E$ and $v\in W$.) Check that $\phi$ is a homomorphism of $C^*$-algebras, from $\Compact_B(E)\rtimes W$ to the $C^*$-algebra of adjointable operators $\Adjointable_{B\rtimes W}(E\rtimes W)$. For $\xi,\eta\in E$ and $w\in W$ we have $\phi( \ket{\xi}\bra{\eta} w) = \ket{\xi 1_W}\bra{\gamma_{w^{-1}}\eta w^{-1}}\in \Compact_{B\rtimes W}(E\rtimes W)$, which shows that $\phi$ maps $\Compact_B(E)\rtimes W$ onto $\Compact_{B\rtimes W}(E\rtimes W)$. The map $\phi$ is also injective, because if $a=\sum_{w\in W} k_w w\in \ker \phi$ then for each $w\in W$ and each $\xi\in E$ we have $0 =\phi(a)(\xi 1_W) = \sum_w k_w (\gamma_w \xi) w$, showing that $k_w\circ \gamma_w=0$ for all $w$; and since the $\gamma_w$s are invertible we conclude that $k_w=0$ for all $w$.
\end{proof}

We will now see that in addition to forming the crossed product $E\rtimes W$, one can turn the equivariant Hilbert module $E$ itself into a Hilbert $B\rtimes W$-module. The two constructions are closely related, as we shall see in the proof of Theorem \ref{thm:Green-Julg}.

\begin{definition}\label{def:Green-Julg}
Let $E$ be a $W$-equivariant Hilbert $B$-module. We make $E$ into a right module over $B\rtimes W$ by defining $\xi bw \coloneqq \gamma_{w^{-1}}(\xi b) = \gamma_{w^{-1}}(\xi) \beta_{w^{-1}}(b)$, and we equip $E$ with the $B\rtimes W$-valued inner product 
\[
\llangle \xi\, |\, \eta\rrangle \coloneqq \sum_{w\in W}\langle\xi\ |\ \gamma_w\eta\rangle w.
\]
\end{definition}

\begin{theorem}\label{thm:Green-Julg}
The above definitions make $E$ into a Hilbert $B\rtimes W$-module, and we have 
\[
\Compact_{B\rtimes W}(E) = \Compact_B(E)^W = \{k\in \Compact_B(E)\ |\ \gamma_w k=k \gamma_w \text{ for all }w\in W\}.
\]
\end{theorem}

\begin{proof}[Proof (outline)]
Let $E\rtimes W$ be the Hilbert $B\rtimes W$-module from Theorem \ref{thm:E-cross-W}. The map
\[
\psi : E\to E\rtimes W,\qquad \psi(\xi) = |W|^{-1/2}\sum_{w\in W} \gamma_{w^{-1}}(\xi)w
\]
satisfies $\psi(\xi bw)=\psi(\xi)bw$ and $\llangle \xi\, |\, \eta\rrangle = \langle \psi(\xi)\, |\, \psi(\eta)\rangle$ for all $\xi,\eta\in E$, $b\in B$, and $w\in W$. Thus $E$, equipped with the multiplication and the inner product from Definition \ref{def:Green-Julg}, is unitarily isomorphic to a $B\rtimes W$-submodule of the Hilbert $B$-module $B\rtimes W$.   We now just need to check that $E$ is complete in the norm $\|\ \|_{\llangle\ |\ \rrangle}$ induced by the inner product $\llangle\ |\ \rrangle$. To do this we show that this norm is equivalent to the original norm $\|\ \|$ on $E$. On the one hand we have
\[
\|\xi\|^2_{\llangle\ |\ \rrangle} = \left\| \textstyle\sum_{w\in W} \langle \xi\, |\, \gamma_w\xi\rangle w \right\| \leq |W| \|\xi\|^2
\]
(using  the Cauchy--Schwarz inequality, and the fact that the $\gamma_w$s are isometric.) On the other hand, take an arbitrary $\xi\in E$ with $\|\xi\|=1$. Since the element $\langle\xi\, |\, \xi\rangle\in B$ is positive we can write it as $aa^*$ for some $a\in B$. We then have $\|a\|=\|aa^*\|^{1/2}=\|\xi\|=1$. Using the embedding $t:B\rtimes W\into \Adjointable_B(\ell^2 W\otimes B)$ from Exercise \ref{ex:crossed-product-norm}, the {Cauchy--Schwarz inequality} gives 
\[
\|\xi \|_{\llangle\ |\ \rrangle}^2 \geq \| \langle \delta_{1_W}\otimes a\, |\, t_{\llangle \xi\, |\, \xi \rrangle}(\delta_{1_W}\otimes a) \rangle\|= \| a^*\langle \xi\, |\, \xi\rangle a\|=\|a\|^4=1,
\]
showing that $\|\xi\|_{\llangle\ |\ \rrangle}\geq \|\xi\|$. 

Now for the assertion about compact operators:  for all $\xi,\eta\in E$ the operator $|\xi\rrangle\llangle \eta|\in \Compact_{B\rtimes W}(E)$ is equal to $\sum_{w\in W} \gamma_{w^{-1}} \ket{\xi}\bra{\eta} \gamma_w$, which clearly lies in $\Compact_B(E)^W$. Taking closed spans shows that $\Compact_{B\rtimes W}(E)\subseteq \Compact_B(E)^W$. On the other hand, the same computation shows that for every $k\in \Compact_B(E)$ we have $\sum_{w\in W} \gamma_{w^{-1}} k \gamma_w \in \Compact_{B\rtimes W}(E)$, and for every $k\in \Compact_B(E)^W$ we have $\sum_{w\in W} \gamma_{w^{-1}} k \gamma_w = |W|k$. So $\Compact_B(E)^W\subseteq \Compact_{B\rtimes W}(E)$.
\end{proof}

\begin{remark}
Theorem \ref{thm:Green-Julg} is closely related to the \emph{Green--Julg theorem} \cite{Julg}, which says that the construction in Definition \ref{def:Green-Julg} (and its generalisation to the case where $W$ is a compact group) gives an isomorphism from the $W$-equivariant $K$-theory of $B$, to the $K$-theory of $B\rtimes W$.
\end{remark}

\subsection{A Morita equivalence for $C_0(X,\Compact(H))^W$}\label{sec:CXKW-Morita}

Consider again the situation of Section \ref{subsec:CXKW}: $X$ is a locally compact Hausdorff space, $W$ is a finite group acting on $X$ by homeomorphisms, $H$ is a Hilbert space, and we are given a strong-operator-continuous family of unitary operators $I = \{ I_{w,x}\in \Unitary(H)\ |\ w\in W,\ x\in X\}$ satisfying $I_{w_1,w_2x}I_{w_2,x}=I_{w_1w_2,x}$. As explained in Example \ref{example:CXH-equivariant}, we can use these operators to view $C_0(X,H)$ as a $W$-equivariant Hilbert $C_0(X)$-module. Theorem \ref{thm:Green-Julg} shows how to turn this equivariant Hilbert module into a Hilbert module over $C_0(X)\rtimes W$, in such a way that 
\[
\Compact_{C_0(X)\rtimes W}(C_0(X,H)) = \Compact_{C_0(X)}(C_0(X,H))^W \cong C_0(X,\Compact(H))^W
\]
(\emph{cf}.~Exercise \ref{ex:K-C_0-H}), where $W$ acts on $C_0(X,\Compact(H))$ as in \eqref{eq:alpha-def}.

Let $J\subseteq C_0(X)\rtimes W$ denote the closed ideal $\overline{\lspan}\{ \llangle \xi\, |\, \eta\rrangle\ |\ \xi,\eta\in 
C_0(X,H)\}$. The module $C_0(X,H)$ then implements a Morita equivalence $C_0(X,\Compact(H))^W \sim J$. Our goal is to describe this ideal $J$, under some additional hypotheses that will be satisfied in the situations arising in the study of $C^*_r(G)$ for a real reductive group $G$. 

To motivate the hypotheses, first suppose that all of the operators $I_{w,x}\in \Unitary(H)$ are scalar multiples of the identity; this will almost never happen in the setting of reductive groups, but the experiment is enlightening nonetheless. The action $\alpha$ of $W$ on $C_0(X,\Compact(H))$ simplifies in this case to $\alpha_w(k)(x) = k(w^{-1}x)$, and we have $C_0(X,\Compact(H))^W\cong C_0(X/W,\Compact(H))$, which we know to be Morita equivalent to the commutative $C^*$-algebra $C_0(X/W)$. 

This result cannot hold in general, because Morita equivalence preserves spectra, and $\widehat{C_0(X/W)}\cong X/W$ is Hausdorff, while we know (\emph{e.g.,}, from Exercise \ref{ex:CXKW-example}) that the spectrum of $C_0(X,\Compact(H))^W$ can be non-Hausdorff. But our thought experiment suggests, at least, that it would be helpful to pay attention to the question of which $I_{w,x}$s are scalars.

\begin{definition}
Given $X$, $W$, $H$, and $I$ as above, for each $x\in X$ we define
\[
W'_x = \{w\in W_x\ |\ \text{the operator $I_{w,x}$ is a scalar multiple of the identity $\id_H$}\}.
\]
\end{definition}

It is easily checked that for each $w\in W$ and each $x\in X$ we have $wW'_x w^{-1}=W'_{wx}$. In particular, $W'_x$ is a normal subgroup of $W_x$.

We will work under the following additional assumptions:

\begin{definition}\label{def:I-conditions}
\begin{enumerate}[\rm(1)]
\item We say that the family of operators $I$ satisfies the \emph{normalisation condition} if for each $x\in X$ and each $w\in W'_x$ we have $I_{w,x}=\id_H$.
\item If $I$ satisfies the normalisation condition, then we say that it also satisfies the \emph{completeness condition} if for each $x\in X$ the representation $W_x\to \Unitary(H)$ defined by $w\mapsto I_{w,x}$ contains a copy of each irreducible unitary representation of $W_x$ that is trivial on $W'_x$.
\end{enumerate}
\end{definition}

\begin{definition}\label{def:CXWI}
Given $X$, $W$, $H$, and $I$ as above, we consider the following subset of $C_0(X)\rtimes W$:
\[
C(X,W,I) = \left\{ \sum_{w\in W} f_w w\in C_0(X)\rtimes W\ \middle|\  \begin{aligned} &\forall x\in X,\ \forall w'\in W_x',\ \forall w\in W:\\ & \quad \quad \quad f_{w'w}(x)  = f_w(x)\end{aligned} \right\}.
\]
\end{definition}
Note that $C(X,W,I)$ depends on $H$ and $I$ only through the system of subgroups $W'_x\subseteq W_x$ that they define.

\begin{exercise}\label{ex:C-ideal}
Prove that $C(X,W,I)$ is an ideal in $C_0(X)\rtimes W$. (\emph{Hint:} first consider the case where $W$ acts transitively on $X$; then use the fact that pre-images of ideals under homomorphisms, and intersections of ideals, are ideals.)
\end{exercise}

\begin{theorem}\label{thm:CXKW-Morita}
If $I$ satisfies the completeness and the normalisation conditions, then the $C^*$-algebra $C_0(X,\Compact(H))^W$ is Morita equivalent to $C(X,W,I)$.
\end{theorem}

\begin{proof}[Proof (outline)]
We noted above that Theorem \ref{thm:Green-Julg}, applied to the $W$-equivariant Hilbert $C_0(X)$-module $C_0(X,H)$, gives a Morita equivalence between $C_0(X,\Compact(H))^W$ and the ideal
\[
J = \overline{\lspan}\{ \llangle \xi\, |\, \eta\rrangle\in C_0(X)\rtimes W\ |\ \xi,\eta\in C_0(X,H)\}
\]
in $C_0(X)\rtimes W$. Our goal is to prove that $J=C(X,W,I)$. 

To save space we write $C=C(X,W,I)$. It is easy to check that $\llangle \xi \,|\,\eta\rrangle \in C$ for all $\xi,\eta\in C_0(X,H)$, and so $J\subseteq C$. Since both $J$ and $C$ are ideals in $C_0(X)\rtimes W$, to prove that they are equal it will suffice to prove that every irreducible representation of $C_0(X)\rtimes W$ that annihilates $J$ also annihilates $C$ (Theorem \ref{thm:ideal-irreps}(7).) 

Up to isomorphism, each irreducible representation of $C_0(X)\rtimes W$ is the composition of the isomorphism $\phi:C_0(X)\rtimes W\to C_0(X,\Compact(\ell^2 W))^W$ from Exercise \ref{ex:fixed-point-crossed-product}, with one of the irreducible representations $\pi_{x,\rho}:C_0(X,\Compact(\ell^2 W))^W\to \Bounded(\HS(H_\rho,\ell^2 W)^{W_x})$ from Theorem \ref{thm:fixed-irreps}. Given a point $x\in X$ and an irreducible representation $\rho\in \widehat{W_x}$ on which $W'_x$ acts trivially, one can use the completeness condition to cook up a function $\xi\in C_0(X,H)$ with $\pi_{x,\rho}(\phi(\llangle \xi\, |\, \xi\rrangle))\neq 0$. Here is one way to do this. The completeness condition implies that we can find a unitary operator $u:H_\rho\to H$ satisfying $I_{w,x}\circ u = u\circ \rho(w)$ for all $w\in W_x$. Given a nonzero $s\in \HS(H_\rho,\ell^2 W)^{W_x}$, let $\zeta\coloneqq s^*\delta_1\in H_\rho$. We have $\zeta\neq 0$, because there is at least one $w\in W_x$ and at least one $\eta\in H_\rho$ with $\langle \delta_w\, |\, s\eta\rangle\neq 0$, and since $s$ is $W_x$-equivariant we have $\langle \delta_w\, |\, s\eta\rangle = \langle s^*\delta_1\, |\ \rho(w)\eta\rangle$. Now let $\xi\in C_0(X,H)$ be a function with $\xi(x)=\zeta$, and with $\xi(y)=0$ for each $y\neq x$ in the orbit $Wx$. Unravelling the definitions leads to the equality
\[
\Big\langle \delta_1\, \Big|\, (\pi_{x,\rho}\circ \phi)( \llangle \xi\, |\, \xi\rrangle) (s)(\zeta)\Big\rangle = \sum_{w\in W_x} |\langle\zeta\, |\, \rho(w)\zeta\rangle|^2 >0,
\]
and so the irreducible representation $\pi_{x,\rho}\circ\phi$ of $C_0(X)\rtimes W$ is nonzero on the ideal $J$. In other words,  if an irreducible representation $\pi_{x,\rho}\circ \phi$ of $C_0(X)\rtimes W$ annihilates $J$, then $\rho$ is nontrivial on $W'_x$. 

On the other hand, a straightforward computation shows that for each $c\in C$, $x\in X$, and $w'\in W'_x$, we have $\phi(c)(x)=\lambda_{w'}\phi(c)(x)$ in $\Compact(\ell^2 W)$, and this implies that if $\rho\in \widehat{W_x}$ is nontrivial on $W'_x$ then $\pi_{x,\rho}\circ\phi$ annihilates $C$. Thus every irreducible representation of $C_0(X)\rtimes W$ that annihilates $J$ also annihilates $C$, and so $J=C$.
\end{proof}

\section{The reduced $C^*$-algebra of a real reductive group II}

Let $G$ be a real reductive group. Theorem \ref{thm:Plancherel} (together with Exercise \ref{ex:Morita-sum}) shows that in order to understand $C^*_r(G)$ up to Morita equivalence, it will suffice to consider one of the direct-summands $C_0(\germ{a}_L^*,\Compact(\Ind_P^G H_\sigma))^{W_\sigma}$. So we now fix a Levi subgroup $L$ of a parabolic subgroup $P$ of $G$, and a discrete-series representation $\sigma$ of the compactly generated part of $L$. For each $\chi\in \germ{a}_L^*$ we consider the normal subgroup  \[
W'_{\sigma,\chi}\coloneqq \{w\in W_{\sigma,\chi}\ |\ \text{ the operator $I_{w,\chi}\in \Unitary(\Ind_P^G H_\sigma)$ is a scalar}\}
\]
of the stabiliser of $\chi$ in $W_\sigma$. We let $W_\sigma'\coloneqq W_{\sigma,0}'$.

Our computation relies on the following theorem of Knapp and Stein. 

\begin{theorem}[{\cite[Theorem 13.4 \& Lemma 14.1]{Knapp-Stein-II}}]\label{thm:KS}\hfill
\begin{enumerate}[\rm(1)]
\item The intertwining operators $I_{w,\chi}$ can be normalised so that $I_{w,\chi}=\id_{\Ind_P^G H_\sigma}$ for all $\chi\in \germ{a}_L^*$ and all $w\in W_{\sigma,\chi}'$.
\item Assuming that the $I_{w,\chi}$s are normalised as in part (1), for each $\chi\in \germ{a}_L^*$ the representation $W_{\sigma,\chi}\xrightarrow{w\mapsto I_{w,\chi}} \Unitary(\Ind_P^G H_\sigma)$ contains a copy of each irreducible unitary representation of $W_{\sigma,\chi}$ that is trivial on $W_{\sigma,\chi}'$.
\item For each $\chi\in \germ{a}_L^*$ we have $W_{\sigma,\chi}' = W_{\sigma,\chi}\cap W_{\sigma}'$.
\item There is a subgroup $R_\sigma\subseteq W_\sigma$ such that $W_\sigma = W'_\sigma \rtimes R_\sigma$.\hfill\qed
\end{enumerate}
\end{theorem}

We are now ready to complete our computation of $C^*_r(G)$, up to Morita equivalence:

\begin{theorem}\label{thm:G-Morita}
For each direct-summand in Theorem \ref{thm:Plancherel} there is a Morita equivalence
\[
C_0(\germ{a}_L^*, \Compact(\Ind_P^G H_\sigma))^{W_\sigma} \sim C_0(\germ{a}_L^*/W_\sigma')\rtimes R_\sigma.
\]
\end{theorem}

\begin{proof}[Proof (outline)]
Parts (1) and (2) of Theorem \ref{thm:KS} say that the family of intertwining operators $I_{w,\chi}$ satisfies the normalisation and completeness conditions of Definition \ref{def:I-conditions}, and so Theorem \ref{thm:CXKW-Morita} implies that the Hilbert $C_0(\germ{a}_L^*)\rtimes W_\sigma$-module $C_0(\germ{a}_L^*,\Ind_P^G H_\sigma)$ furnishes a Morita equivalence between the $C^*$-algebra $C_0(\germ{a}_L^*,\Compact(\Ind_P^G H_\sigma))^{W_\sigma}$ and the ideal
\[
C(\germ{a}_L^*,W_\sigma,I) = \left\{ \sum_{w\in W_\sigma} f_w w\in C_0(\germ{a}_L^*)\rtimes W_\sigma\ \middle|\  \begin{aligned} &\forall \chi\in \germ{a}_L^*,\, \forall w'\in W_{\sigma,\chi}',\, \forall w\in W_\sigma:\\ & \quad \quad \quad f_{w'w}(\chi)  = f_w(\chi)\end{aligned} \right\}
\]
in $C_0(\germ{a}_L^*)\rtimes W_\sigma$.

The decomposition $W_\sigma = W'_\sigma \rtimes R_\sigma$ gives an isomorphism $C_0(\germ{a}_L^*)\rtimes W_\sigma\cong(C_0(\germ{a}_L^*)\rtimes W'_\sigma)\rtimes R_\sigma$ (Exercise \ref{ex:iterated-crossed-product}). The fact that $W'_{\sigma,\chi}=W_{\sigma,\chi}\cap W_\sigma'$ for every $\chi$ implies that this isomorphism restricts to an isomorphism between the ideal $C(\germ{a}_L^*,W_\sigma,I)\subseteq C_0(\germ{a}_L^*)\rtimes W_\sigma$ and the ideal $C(\germ{a}_L^*,W_\sigma',I)\rtimes R_\sigma \subseteq (C_0(\germ{a}_L^*)\rtimes W_\sigma')\rtimes R_\sigma$. 

The $C^*$-algebra $C(\germ{a}_L^*,W_\sigma',I)$ depends only on which of the operators $I_{w,\chi}$ are scalars, for $w\in W_\sigma'\cap W_{\sigma,\chi}=W'_{\sigma,\chi}$. The answer in this case is, all of them are scalars (because that is the definition of $W'_{\sigma,\chi}$.) This is the same answer as for the operators $\id_{w,\chi}\coloneqq \id$ on the Hilbert space $\C$, so $C(\germ{a}_L^*,W_\sigma',I)=C(\germ{a}_L^*,W_\sigma',\id)$. Theorem \ref{thm:CXKW-Morita} implies that the  bimodule $C_0(\germ{a}_L^*,\C)$ gives a Morita equivalence between $C(\germ{a}_L^*,W_\sigma',\id)$ and $C_0(\germ{a}_L^*,\Compact(\C))^{W_\sigma'}\cong C_0(\germ{a}_L^*/W_\sigma')$.

The group $R_\sigma$ acts compatibly on $C(\germ{a}_L^*,W_\sigma',\id)$, $C_0(\germ{a}_L^*,\C)$, and $C_0(\germ{a}_L^*/W_\sigma')$, and so the construction of Theorem \ref{thm:E-cross-W} allows us to promote the Morita equivalence $C_0(\germ{a}_L^*,\C)$ to a Morita equivalence $C_0(\germ{a}_L^*,\C)\rtimes R_\sigma$, between $C(\germ{a}_L^*,W_\sigma',\id)\rtimes R_\sigma$ and $C_0(\germ{a}_L^*/W_\sigma')\rtimes R_\sigma$. Putting everything together, we have
\begin{align*}
C_0(\germ{a}_L^*,\Compact(\Ind_P^G H_\sigma))^{W_\sigma} \sim C(\germ{a}_L^*,W_\sigma,I)  \cong 
C(\germ{a}_L^*,W_\sigma',\id)\rtimes R_\sigma \sim C_0(\germ{a}_L^*/W_\sigma')\rtimes R_\sigma.
\end{align*}
\end{proof}

Putting Theorems \ref{thm:Plancherel} and \ref{thm:G-Morita} together gives Theorem \ref{thm:main}.

\appendix

\section{Representations of compact groups}\label{sec:compact-groups}

In this appendix we list some facts about unitary representations of compact (eg, finite) groups that are used in our computation of the reduced $C^*$-algebras of real reductive groups. References for this material include the books of Serre \cite{Serre-en} (for the finite case) and Br\"ocker--tom Dieck \cite[Chapters II--III]{BtD}, among many others.

Let $W$ be a compact topological group, and choose a Haar measure $\dd w$. (So $f\mapsto \int_W f(w)\,\dd w$ is a positive linear map $C(W)\to \C$, satisfying $\int_W f(w_1ww_2)\, \dd w = \int_W f(w)\, \dd w$.) We have:
\begin{enumerate}[\rm(1)]
\item Every irreducible unitary representation of $W$ is finite-dimensional.
\item Every unitary representation of $W$ is isomorphic to a direct sum of irreducible representations.
\item The regular representation $\lambda:G\to L^2(G)$ contains a copy of each irreducible representation of $G$.
\item For each irreducible unitary representation $\rho:W\to \Unitary(H_\rho)$ of $W$ we let $e_\rho\in C(W)$ be the function 
\begin{equation}\label{eq:e-rho}
e_\rho(w) = \frac{\dim H_\rho}{\operatorname{volume}(W)} \overline{\trace(\rho(w))}.
\end{equation}
\item Let $\pi:W\to \Unitary(H_\pi)$ be a unitary representation. For each irreducible representation $\rho$, the operator $\pi(e_\rho)\in \Bounded(H_\pi)$ is an orthogonal projection that commutes with $\rho(w)$ for every $w\in W$. The  range $H_{\pi,\rho}\coloneqq \pi(e_\rho)H_\pi$ of this projection is the largest subspace of $H_\pi$ that is isomorphic, as a representation of $W$, to a direct sum of copies of $\rho$. This $H_{\pi,\rho}$ is called the \emph{$\rho$-isotypical subspace} of $H_\pi$. We say that $\rho$ \emph{occurs in} or \emph{is contained in} $\pi$, and write $\rho\subseteq \pi$, if $H_{\pi,\rho}\neq 0$.
\item We have $H_{\pi,\rho}\perp H_{\pi,\sigma}$ if $\rho$ and $\sigma$ are inequivalent irreducible representations, and $H_\pi = \bigoplus_{\rho\in \widehat{W}} H_{\pi,\rho}$. (The direct sum is over a set of representatives for the equivalence classes of irreducible unitary representations of $W$.)
\item Fix a representation $\pi$ and an irreducible representation $\rho$. Let $\HS(\rho,\pi)^W$ denote the space of linear maps $s:H_\rho\to H_\pi$ satisfying $s\rho(w)=\pi(w)s$ for all $w\in W$. Since $\dim H_\rho<\infty$, this space of maps is a Hilbert space under the Hilbert--Schmidt inner product $\langle r\,|\,s\rangle = \trace(r^*s)$. The map
\[
\mu_{\pi,\rho}: H_\rho \otimes \HS(\rho,\pi)^W \to H_{\pi,\rho},\qquad \zeta\otimes s\mapsto (\dim H_\rho)^{1/2} s(\zeta)
\]
is a unitary isomorphism, satisfying
\[
\mu_{\pi,\rho}\circ(\rho(w) \otimes \id_{\HS(\rho,\pi)^W}) = \pi(w)\circ \mu_{\pi,\rho}
\]
for all $w\in W$.
\item Let $\rho:W\to \Unitary(H_\rho)$ be an irreducible representation of $W$. For vectors $\xi,\eta\in H_\rho$ we consider the \emph{matrix coefficient} function
\[
c_{\xi,\eta}:W\to \C,\qquad c_{\xi,\eta}(w)\coloneqq \langle \rho(w)\xi\,|\, \eta\rangle.
\]
The $c_{\xi,\eta}$s are continuous, and thus in particular square-integrable. The \emph{Schur orthogonality relations} say that
\[
\langle c_{\xi,\eta}\,|\,c_{\xi',\eta'}\rangle_{L^2(W)} = (\dim H_\rho)^{-1} \overline{\langle \xi\,|\,\xi'\rangle_{H_\rho}} \langle \eta\,|\,\eta'\rangle_{H_\rho}.
\]

\end{enumerate}

\section{Solutions to the exercises}\label{sec:solutions}

 \begin{proof}[Exercise \ref{ex:W-acting-on-X}]
For each $f\in C_0(X)^W$ let $\phi(f):X/W\to \C$ be defined by $f(Wx)\coloneqq f(x)$. Show that $\phi(f)\in C_0(X/W)$, and that $\phi:C_0(X)\to C_0(X/W)$ is a homomorphism of $C^*$-algebras. If $\phi(f)=0$ then $f=0$, obviously. If $Wx\neq Wy$ then by Urysohn's lemma we can find $f\in C_0(X)$ with $f(wx)=1$ for all $w\in W$, and $f(wy)=0$ for all $w\in W$. Let $h=\sum_{w\in W}\beta_w(f)$. This function is $W$-invariant, and it has $\phi(h)(Wx)\neq 0$ and $\phi(h)(Wy)=0$. So the Stone--Weierstrass theorem implies that $\phi$ is surjective.
\end{proof}

\begin{proof}[Exercise \ref{ex:iterated-crossed-product}]
Routine computations confirm that each $\alpha_v$ is an automorphism of the $C^*$-algebra $A\rtimes U$; that $\alpha_{v_1 v_2}=\alpha_{v_1}\alpha_{v_2}$; and that the linear map
 \[
 \phi : (A\rtimes U)\rtimes V \to A\rtimes W,\qquad \phi\big( (au)v \big)\coloneqq a(uv)
 \]
 is an isomorphism of $C^*$-algebras.
 \end{proof}

\begin{proof}[Exercise \ref{ex:representation-proofs}]
(1) Suppose that $\pi(J)\neq 0$. The set $K=\{\xi\in H\ |\ \pi(j)\xi=0\text{ for all }j\in J\}$ is a closed, $A$-invariant subspace of $H$. Since $\pi$ is irreducible and $\pi(J)\neq 0$ we must have $K=0$. Now let $H'$ be a closed, $J$-invariant subspace of $H$. The set $JH'$ is then a closed $A$-invariant subspace of $H$, so either $JH'=H$ or $JH'=0$. If $JH'=H$ then $H'=H$, while if $JH'=0$ then $H'\subseteq K=0$ and so $H'=0$. Thus $\pi\restrict_J$ is irreducible.

(2) Part (3) of Theorem \ref{thm:reps} implies that $\rho$ is the restriction of an irreducible representation $\pi:A\to \Bounded(H')$ to the $J$-invariant subspace $H\subseteq H'$. Since $\pi\restrict_J$ is irreducible we must have $H=H'$, so $\rho=\pi\restrict_{J}$. 

(3) Let $u:H_\rho\to H_\pi$ be a unitary with $u\rho(j)=\pi(j)u$ for all $j\in J$. The factorisation theorem (Theorem \ref{thm:Cohen-Hewitt}) implies that every $\xi\in H_\rho$ can be written as $\rho(j)\eta$ for some $j\in J$ and $\eta\in H_\rho$, and for each $a\in A$ we have
\[
u\rho(a)\xi = u\rho(aj)\eta = \pi(aj)u\eta = \pi(a)\pi(j)u\eta=\pi(a)u\rho(j)\eta=\pi(a)u\xi.
\]
Thus $u$ is an equivalence between $\pi$ and $\rho$ as representations of $A$. The other implication is easy.

(4) Parts (1)--(3) show that restriction to $J$ gives a bijective map $O_J\to \widehat{J}$. The open subsets of $O_J$ are those $O_{J'}$s with $O_{J'}\subseteq O_J$, while the open subsets of $\widehat{J}$ are those $O_{J'}$ with $J'\subseteq J$. Prove that the restriction map is a homeomorphism by proving that $O_{J'}\subseteq O_J$ if and only if $J'\subseteq J$. The definitions immediately imply that if $J'\subseteq J$ then $O_{J'}\subseteq O_J$. For the reverse implication, suppose that $J'\not\subseteq J$. Then the nonzero $C^*$-algebra $J'/(J'\cap J)$ has an irreducible representation $\pi$. Since $J'/(J'\cap J)$ is an ideal in $A/J$, $\pi$ extends to an irreducible representation of $A/J$, and pulling back to $A$ we have an irreducible $\pi\in \widehat{A}$ with $\pi\restrict_J=0$ and $\pi\restrict_{J'}\neq 0$; so $\pi\in O_{J'}\setminus O_J$, and thus $O_{J'}\not\subseteq O_J$.  

(5) Follows immediately from the isomorphism theorems.

(6) Part (5) implies that the given map is a bijection. The open sets in $\widehat{A}\setminus O_J$ are those of the form $O_{J'}\setminus O_J$, where $J'$ is an ideal of $A$ such that $O_J\subseteq O_{J'}$. We observed above that this condition is equivalent to $J\subseteq J'$. On the other hand, the open subsets of $\widehat{A/J}$ are those of the form
\[
O_{K} = \left\{\widetilde{\pi}\in\widehat{A/J}\ \middle|\ \widetilde{\pi}\restrict_K\neq 0\right\}
\]
for ideals $K\subseteq A/J$. One of the isomorphism theorems says that the map $J'\mapsto J'/J$ is a bijection between the set of ideals $J'\subseteq A$ that contain $J$, and the set of ideals of $A/J$.  Thus the map $\pi\mapsto \widetilde{\pi}$ induces a bijection between the sets of open subsets of $\widehat{A}\setminus O_J$ and $\widehat{A/J}$.

(7) If $J=A$ then no irreducible representation annihilates $J$, because irreducible representations are by definition nonzero. If $J\neq A$ then the nonzero $C^*$-algebra $A/J$ has an irreducible representation (by Theorem \ref{thm:reps}), and pulling back to $A$ we get an irreducible representation that annihilates $J$. 

(8) Let $K$ denote the given intersection. Obviously $J$ is contained in $K$. For the reverse inclusion, let $\pi$ be an irreducible (hence nonzero) representation of $K$. Part (2) implies that $\pi$ extends to an irreducible representation of $A$, and this representation cannot annihilate $J$ (because otherwise the definition of $K$ would give $K\subseteq \ker\pi$.) So part (7) implies that $J=K$.
\end{proof}

\begin{proof}[Exercise \ref{ex:CX-spectrum}]
Recall that to each ideal $J\subseteq C_0(X)$ we associated an open subset
\[
O_J \coloneqq \left\{\ev_x\in \widehat{C_0(X)}\ \middle|\ \exists j\in J\text{ with }\ev_x(j)\neq 0\right\}.
\]
We have
\[
\ev^{-1}(O_J) = \{x\in X\ |\ \exists j\in J\text{ with }j(x)\neq 0\}.
\]
This is an open subset of $X$, because each $j\in J$ is a continuous function. Thus the map $\ev:X\to \widehat{C_0(X)}$ is continuous.

It remains to show that $\ev$ is open. Let $U\subseteq X$ be an open set, and let $J\coloneqq \{j\in C_0(X)\ |\ j\restrict_{X\setminus U}=0\}$. This $J$ is an ideal in $C_0(X)$---it is the kernel of the restriction homomorphism $C_0(X)\to C_0(X\setminus U)$---and we claim that $\ev(U) = O_J$. If $x\in U$ then we can find some $j\in J$ with $j(x)\neq 0$, by Urysohn's lemma. So $\ev_x\in O_J$. Conversely, if $\ev_x\in O_J$ then there is some $j\in J$ with $j(x)\neq 0$, and looking at the definition of $J$ shows that we must have $x\in U$.
\end{proof}

\begin{proof}[Exercise \ref{ex:K-irreducible}]
Let $V$ be a nonzero $\Compact(H)$-invariant subspace of $H$, let $\eta\in V$ be a unit vector, and for each $\xi\in H$ consider the compact operator $\ket{\xi}\bra{\eta}:\zeta\mapsto \langle \eta\, |\, \zeta\rangle\xi$. We have $\xi=\ket{\xi}\bra{\eta}(\eta)\in \Compact(H)V=V$, so $V=H$.
\end{proof}

\begin{proof}[Exercise \ref{ex:direct-sum-irreps}]
For (1): Let $A=\bigoplus_{i\in I}A_i$, and let $\pi:A\to \Bounded(H)$ be an irreducible representation. For each $i$ we have a canonical embedding $A_i\into A$, and the set $H_i\coloneqq \pi(A_i)H$ is an $A$-invariant subspace of $H$---so either $H_i=0$ or $H_i=H$. We have $H_i\perp H_j$ when $i\neq j$, so at most one of the $H_i$s is nonzero. Since $\pi(A)\neq 0$, we have $\pi(A_i)\neq 0$ for exactly one $i$. Then $\pi$ factors through the projection $A\to A_i$.

For (2): If $u:H\to H'$ is an equivalence between $\pi$ and some other irreducible representation $\pi'$, then $u$ restricts to a unitary isomorphism $H_i\to H'_i$ for each $i$, which  implies that $\pi$ and $\pi'$ factor through the same summand $A_i$, and that $u$ is an equivalence between the resulting representations of $A_i$.
\end{proof}

\begin{proof}[Exercise \ref{ex:KW}]
Schur orthogonality ensures that for each $s\in \HS(\rho,\pi)^W$ we have $s(H_\rho)\subseteq H_{\pi,\rho}$, so the composition $t\circ s$ (for $t\in \Bounded(H_{\pi,\rho})^W$) is well defined. Compositions of $W$-equivariant maps are $W$-equivariant, so $ts\in \HS(\rho,\pi)^W$. We have
\[
\| ts\|_{\HS}^2 = \trace(s^*t^*ts) = \trace(t^*tss^*) \leq \|t^*t\|\trace(s^*s) = \|t\|^2 \|s\|_{\HS}^2,
\]
showing that the map $\phi_{\pi,\rho}(t):s\mapsto ts$ is bounded on $\HS(\rho,\pi)^W$. An easy computation shows that $\phi_{\pi,\rho}$ is a homomorphism of $C^*$-algebras. 

Fix $t\in \Bounded(H_{\pi,\rho})^W$. For each $\zeta\in H_\rho$ and $s\in \HS(\rho,\pi)^W$ we have
\[
t \mu_{\pi,\rho}(\zeta\otimes s) = (\dim H_\rho)^{1/2} t s(\zeta) = \mu_{\pi,\rho}(\zeta\otimes \phi_{\pi,\rho}(t)s).
\]
If $\phi_{\pi,\rho}(t)=0$ we therefore have $t\mu_{\pi,\rho}=0$, and since $\mu_{\pi,\rho}$ is a unitary isomorphism onto $H_{\pi,\rho}$ this implies that $t=0$. Thus $\phi_{\pi,\rho}$ is injective.

Next we will show that $\phi_{\pi,\rho}(k)$ is a compact operator on $\HS(\rho,\pi)^W$ when $k\in \Compact(H_{\pi,\rho})^W$. To do this we will first ignore the group $W$, and consider the Hilbert space $\HS(\rho,\pi)$ of all linear maps $H_\rho\to H$, with the Hilbert--Schmidt inner product. Each bounded operator $t\in \Bounded(H)$ defines a linear map $\psi(t)\in \Bounded(\HS(\rho,\pi))$ by composition, and the map $\psi:\Bounded(H)\to \Bounded(\HS(\rho,\pi))$ is a homomorphism of $C^*$-algebras. Given a rank-one operator $k=\ket{\xi}\bra{\eta}$, where $\xi,\eta\in H$, the map $\psi(k)$ factors as the composition
\[
\HS(\rho,\pi) \xrightarrow{s\mapsto \bra{s^*\eta}} H_\rho^* \xrightarrow{\bra{\zeta}\mapsto \ket{\eta}\bra{\zeta}} \HS(\rho,\pi).
\]
The Hilbert space $H_\rho^*$ is finite-dimensional, so $\psi(k)$ has finite rank. By continuity, it follows that $\psi(\Compact(H))\subseteq \Compact(\HS(\rho,\pi))$. In particular, for each $k\in \Compact(H_{\pi,\rho})^W$ the operator $\psi(k)$ is compact, and so the restriction of this operator to the subspace $\HS(\rho,\pi)^W\subseteq \HS(\rho,\pi)$ is also compact. This restriction is precisely $\phi_{\pi,\rho}(k)$. So $\phi_{\pi,\rho}$ sends compact operators to compact operators.

Finally we show that every bounded operator $s$ on $\HS(\rho,\pi)^W$ has the form $\phi_{\pi,\rho}(t)$ for some  $t\in \Bounded(H_{\pi,\rho})^W$, and that this $t$ is compact if $s$ is compact.  Fix $s\in \Bounded(\HS(\rho,\pi)^W)$, and consider the operator $t$ on $H_{\pi,\rho}$ defined as the composition
\[
H_{\pi,\rho}\xrightarrow{\mu_{\pi,\rho}^*} H_\rho \otimes \HS(\rho,\pi)^W \xrightarrow{\id_{H_\rho}\otimes s} H_\rho\otimes \HS(\rho,\pi)^W \xrightarrow{\mu_{\rho,\pi}} H_{\pi,\rho}.
\]
The intertwining relation $\mu_{\pi,\rho}\circ(\rho(w)\otimes \id_{\HS(\rho,\pi)^W})= \pi(w)\circ \mu_{\pi,\rho}$ ensures that $t\in \Bounded(H_{\pi,\rho})^W$. Since $H_\rho$ is finite-dimensional, if $s$ is a compact operator then the operator $\id_{H_\rho}\otimes s$ is compact, and so $t$ is compact. For each $r\in \HS(\rho,\pi)^W$ and $\zeta\in H_\rho$ we have
\[
\begin{aligned}
& \mu_{\pi,\rho}^*\phi_{\pi,\rho}(t)(r)\zeta 
=
\mu_{\pi,\rho}^* tr(\zeta)
 =(\id_{H_\rho}\otimes s)\mu_{\pi,\rho}^*(r(\zeta)) \\
& = (\dim H_\rho)^{-1/2}\zeta \otimes s(r)  = \mu_{\pi,\rho}^* s(r)(\zeta).
\end{aligned}
\]
Multiplying on the left by $\mu_{\pi,\rho}$ gives $\phi_{\pi,\rho}(t)(r)\zeta=s(r)(\zeta)$, and since this holds for all $r$ and $\zeta$ we conclude that $\phi_{\pi,\rho}(t)=s$ as required.
\end{proof}

\begin{proof}[Exercise \ref{ex:XH-action}]
First assume that $W$ acts on $(X,H)$. The assumption that the projection $X\times H\to X$ is $W$-equivariant implies that the action of $W$ on $X\times H$ is given by $w(x,\xi)=(wx,I_{w,x}\xi)$. Comparing the $H$-coordinates on both sides of the equality $(w_1 w_2)(x,\xi)=w_1(w_2(x,\xi))$ gives the identity (i), and the map $x\mapsto I_{w,x}\xi$ from (ii) is the composition of the continuous maps
\[
X \xrightarrow{x\mapsto (x,\xi)} X\times H \xrightarrow{\text{act by $w$}} X\times H \xrightarrow{\text{project}} H.
\]

Now, for part (2), assume that we have a family of unitary operators $I_{w,x}$ satisfying conditions (i) and (ii). The same observation as in the proof of part (1) shows that the formula $w(x,\xi)\coloneqq (wx,I_{w,x}\xi)$ defines an action (just at the level of sets) of $W$ on $X\times H$, for which the projection map $X\times H\to X$ is $W$-equivariant. To see that this action is continuous, fix $w\in W$ and let $(x_i,\xi_i)$ be a net in $X\times H$ converging to $(x,\xi)$. We must show that the net $w(x_i,\xi_i) = (wx_i, I_{w,x_i}\xi_i)$ converges to $w(x,\xi)=(wx,I_{w,x}\xi)$. We are assuming that the $W$-action on $X$ is continuous, so certainly $wx_i$ converges to $wx$. In the $H$-coordinate we have
\[
\| I_{w,x_i}\xi_i - I_{w,x}\xi\| \leq \|I_{w,x_i}\xi_i - I_{w,x_i}\xi\| + \|I_{w,x_i}\xi - I_{w,x}\xi\|,
\]
where the first summand converges to $0$ because each $I_{w,x_i}$ is unitary and $\xi_i$ converges to $\xi$; and the second summand converges to $0$ by condition (ii).
\end{proof}

\begin{proof}[Exercise \ref{ex:beta-w-f-continuous}]
Fix $k\in C_0(X,\Compact(H))$, $x\in X$, and $w\in W$. Suppose that $x_i$ is a net in $X$ converging to $x$. To save space, write $y\coloneqq w^{-1}x$ and $y_i\coloneqq w^{-1}x_i$.  

Fix $\epsilon>0$. Since $k$ is continuous, we can find $I_1$ such that $i>I_1$ implies $\| k(y_i)-k(y)\|<\epsilon$. Since $k(y)$ is a compact operator, we can find a finite-rank operator $f=\sum_{m=1}^n \ket{\eta_m}\bra{\xi_m}$ with $\|k(y)-f\|<\epsilon$. Extend $f$ to a $C_0$ function on $X$, with $f(y_i)=f(y)=f$ for all  $i>I_1$.  Note that we then have, for $i>I_1$, 
\[
\|k(y_i)-f\| \leq \|k(y_i)-k(y)\| + \|k(y)-f\| <2\epsilon.
\]
Since the family of unitary operators $I_{w,x}$ is strong-operator continuous, we can find $I_2>I_1$ such that $i>I_2$ implies 
\[
\|(I_{w,y_i}-I_{w,y})\xi_m\|<\frac{\epsilon}{n\|\eta_m\|} \quad \text{and}\quad \|(I_{w,y_i}-I_{w,y})\eta_m\|<\frac{\epsilon}{n\|\xi_m\|}
\]
for all $m$. We have
\begin{multline*}
\| \alpha_w(k)(x_i)-\alpha_w(k)(x)\|  \\ \leq \| \alpha_w(k-f)(x_i)\| + \|\alpha_w(f)(x_i)-\alpha_w(f)(x)\| + \|\alpha_w(f-k)(x)\|.
\end{multline*}
Bounding the three summands on the right-hand side in turn, for $i>I_2$, we find that $\| \alpha_w(k)(x_i)-\alpha_w(k)(x)\| < 5\epsilon$.

This shows that $\alpha_w(k):X\to \Compact(H)$ is a norm-continuous function. Since the operators $I_{w,w^{-1}x}$ and $I_{w^{-1},x}$ are unitary we have $\|\alpha_w(k)(x)\| = \|k(w^{-1}x)\|\to 0$ as $x\to\infty$, so $\alpha_w(k)\in C_0(X,\Compact(H))$. The fact that each $\alpha_w$ is a $C^*$-algebra automorphism, and that $\alpha_{w_1 w_2} = \alpha_{w_1}\circ \alpha_{w_2}$, follows easily from the identity (i) in Exercise \ref{ex:XH-action}.
\end{proof}

\begin{proof}[Exercise \ref{ex:fixed-point-crossed-product}]
 It is easy to check that the map $fw\mapsto \phi(fw)$ is a $*$-homomorphism from $C_0(X)\rtimes W$ to $C_0(X,\Compact(\ell^2 W))^W$. Define a map $\psi: C_0(X,\Compact(\ell^2 W))^W\to C_0(X)\rtimes W$ by
 \[
 \psi(f) = \sum_{w\in W} \langle \delta_1 \ |\  f(\ )\delta_w\rangle w,
 \]
 where $\langle \delta_1 \ |\ f(\ )\delta_w\rangle\in C_0(X)$ means the function $x\mapsto \langle \delta_1\ |\ f(x)\delta_w\rangle$. Check that $\psi$ is an inverse to $\phi$.
\end{proof}

\begin{proof}[Exercise \ref{ex:CXKW-example}]
To check that we have an action on $(\R,\C^2)$ we just need to check that $I_{w,-x}I_{w,x}=I_{1,x}$ for each $x\in \R$, which is clearly true.

For each $x\neq 0$ we have $Wx=\{x,-x\}$, and $W_x=\{1\}$, which has only one irreducible representation (the trivial one-dimensional one). So $C_0(\R,\Compact(\C^2))^W$ has one irreducible representation for each orbit $\{x,-x\}$ with $x\neq 0$. For $x=0$ we have $Wx=\{x\}$ and $W_x=W$, which has two irreducible representations: the trivial one ($w\mapsto 1$) and the sign representation ($w\mapsto -1$). Both of these representations occur in the representation $I_0 : w\mapsto \smallbmat{1 & 0 \\ 0 & {-1}}$, so the orbit $\{0\}$ gives two inequivalent irreducible representations of $C_0(\R,\Compact(\C^2))^W$. So we have
\[
\left(C_0(\R,\Compact(\C^2))^W\right)^{\widehat{\ }}= \{ \pi_{x,\operatorname{triv}}\ |\ x>0\} \sqcup \{ \pi_{0,\operatorname{triv}}, \pi_{0,\operatorname{sign}}\}.
\]
Explicitly, for each $k\in C_0(\R,\Compact(\C^2))^W$ and each $x\geq 0$ we have $\pi_{x,\operatorname{triv}}(k)=k(x)\in \Bounded(\C^2)$, while $\pi_{0,\operatorname{triv}}(k) = k(0)_{1,1}$ (the top-left entry of the diagonal matrix $k(0)$, with respect to the standard basis) and $\pi_{0,\operatorname{sign}}(k)=k(x)_{2,2}$ (the bottom-right entry.)

To show that the dual is non-Hausdorff we will show that the points $\pi_{0,\operatorname{triv}}$ and $\pi_{0,\operatorname{sign}}$ cannot be separated by open sets. Recall that the open sets $O_J$ in $\left(C_0(\R,\Compact(\C^2))^W\right)^{\widehat{\ }}$ correspond to ideals $J$ in $C_0(\R,\Compact(\C^2))^W$: we have $\pi_{x,\rho}\in O_J$ if and only if $\pi_{x,\rho}\restrict_J\neq 0$. Let $J$ be an ideal in $C_0(\R,\Compact(\C^2))^W$ such that $\pi_{0,\operatorname{triv}}\in O_J$. Then $J$ must contain a function $j$ with $j(0)_{1,1}\neq 0$, and this ensures that $j(x)_{1,1}\neq 0$ for all $x$ in some nonempty interval $(0,\epsilon)$. Thus $\pi_{x,\operatorname{triv}}(j)=j(x)\neq 0$ for all $x\in (-\epsilon,\epsilon)$, and so the open set $O_J$ contains $\pi_{x,\operatorname{triv}}$ for all $x\in (0,\epsilon)$. The same argument shows that every open set containing $\pi_{0,\operatorname{sign}}$ contains some nonempty interval $\{\pi_{x,\operatorname{triv}}\ |\ x\in (0,\delta)\}$, and so the points $\pi_{0,\operatorname{triv}}$ and $\pi_{0,\operatorname{sign}}$ do not have disjoint open neighbourhoods.
\end{proof}

\begin{proof}[Exercise \ref{ex:isotypical-approximation}] 
A straightforward estimate shows that for every $h\in C_c(G)$ we have $\|h\|_{C^*_r(G)}=\|\lambda(h)\|_{\Bounded(L^2(G))}\leq \|h\|_{L^1(G)}$, so it will suffice to prove the convergence in the $L^1$ norm. Since $f$ has compact support, and $K$ is compact, there is a compact subset $C$ of $G$ such that $\operatorname{support}(f-e_F\ast f)\subseteq C$ for every $F$, and then the Cauchy--Schwarz inequality  gives $\|f-e_F\ast f\|_{L^1(G)}\leq \operatorname{volume}(C)^{1/2} \|f-e_F\ast f\|_{L^2(G)}$.  So it will suffice to prove convergence in the $L^2$ norm, and this convergence follows from the isotypical decomposition $L^2(G) = \bigoplus_{\rho\in \widehat{K}} L^2(G)_\rho$.
\end{proof}

\begin{proof}[Exercise \ref{ex:separating-irreps}]
Condition (1) in Definition \ref{def:SWP} implies that restriction of representations from $A$ to $B$ gives a well-defined map $r:\widehat{A}\to \widehat{B}$. Condition (2) ensures that this map is injective, while part (3) of Theorem \ref{thm:reps} implies that it is surjective. 

To show that $r$ is a homeomorphism, let $J$ be an ideal in $A$. We claim that for each $\pi\in \widehat{A}$ we have $\pi\restrict_J=0$ if and only if $\pi\restrict_{J\cap B}=0$. One implication is obvious; for the other, note that if $\pi$ annihilates $J\cap B$ then $\pi$ induces an irreducible representation of $B/(J\cap B)$. By one of the isomorphism theorems we have $B/(J\cap B)\subseteq A/J$, so Theorem \ref{thm:reps} implies that there is an irreducible representation $\rho$ of $A$ that annihilates $J$, and whose restriction to $B$ is equivalent to $\pi$. Property (2) from Definition \ref{def:SWP} then ensures that $\rho$ is equivalent to $\pi$ as representations of $A$, and so $J\subseteq\ker\rho=\ker\pi$. 

This claim shows that $r(O_J)=O_{B\cap J}$, where we recall that $O_J$ denotes the open subset $\{\pi\in\widehat{A}\ |\ \pi\restrict_J\neq 0\}$ of $\widehat{A}$. Every ideal of $B$ has the form $B\cap J$ for some ideal $J$ of $A$: given an ideal $K$ of $B$, find a family $(\pi_i)_{i\in I}$ of irreducible representations of $B$ with $\bigcap_{i\in I} (B\cap \ker \pi_i) = K$, and then take $J=\bigcap_{i\in I} \ker \pi_i$. Thus the map $r$ induces a bijection between the set of open subsets of $\widehat{A}$ and the set of open subsets of $\widehat{B}$.
\end{proof}

\begin{proof}[Exercise \ref{ex:trace-k-continuous}]
By linearity it is enough to consider $k_i\to 0$ in the strong-operator topology, and $t$ positive. Write $t = \sum_j c_j \ket{\delta_j}\bra{\delta_j}$ for some orthonormal basis $\{\delta_j\}$ and scalars $c_j\geq 0$. Given $\epsilon>0$, find $n$ such that  $\sum_{j\geq n}^\infty c_j < \epsilon/\sup\|k_i\|$. (Note that the $\|k_i\|$s are bounded, by the uniform boundedness principle.) Then find $I$ such that $i>I$ implies $\|k_i \delta_j\|<\epsilon/(n\sup_l c_l)$ for all $j\leq n$. We have, for $i>I$,
\[
\begin{aligned}
\trace(tk_i)  = \sum_{j=1}^\infty \langle \delta_j\, |\, t k_i\delta_j\rangle 
 = \sum_{j=1}^\infty c_j \langle \delta_j\,|\,k_i\delta_j\rangle  = \sum_{j\leq n} c_j \langle\delta_j\,|\,k_i\delta_j\rangle + \sum_{j>n} c_j\langle \delta_j\,|\,k_i\delta_j\rangle.
\end{aligned}
\]
The absolute value of the first sum on the right-hand side is bounded above by 
\(
\left(\sup_{l\leq n} c_l\right) \sum_{j=1}^n \|k_i\delta_j\|<\epsilon,
\) 
and the absolute value of the second sum is bounded above by 
\(
\left(\sum_{j>n}c_j\right)\sup_i\|k_i\|<\epsilon.
\)
\end{proof}

\begin{proof}[Exercise \ref{ex:SW-extensions}]
Let $B$ be a separating subalgebra of $A$. We claim that $B/(B\cap J)$ is separating in $A/J$, and that $B\cap J$ is separating in $J$. If this claim is true then the Stone--Weierstrass property ensures that $B\cap J=J$ and $B/J=A/J$, and it is easy to conclude from this that $A=B$. 

To prove the claim, we first consider $B/(B\cap J)$ and $A/J$. Let $\pi:A/J\to \Bounded(H)$ be an irreducible representation. The map $a\mapsto \pi(a+J)$ is then an irreducible representation of $A$. Since $B$ is separating in $A$, the map $b\mapsto \pi(b+J)$ is an irreducible representation of $B$, which means that $b+(B\cap J)\mapsto \pi(b+J)$ is an irreducible representation of $B/(B\cap J)$. So property (1) from the definition of `separating' is satisfied. For the second property, note that $\pi$ and $\rho$ are equivalent as representations of $B/(B\cap J)$ if and only  if their pullbacks to $B$ are equivalent, and since $B$ is separating in $A$ this implies that $\pi$ and $\rho$ are equivalent as representations of $A$, and therefore also as representations of $A/J$.

We are left to show that $B\cap J$ is separating in $J$. Recall (Theorem \ref{thm:ideal-irreps}) that every irreducible representation of $J$ is the restriction to $J$ of an irreducible representation of $A$ that does not annihilate $J$. Let $\pi$ be such a representation. The restriction $\pi\restrict_{B\cap J}$ is the restriction to $B\cap J$ of the irreducible representation $\pi\restrict_B$, so $\pi\restrict_{B\cap J}$ is irreducible so long as it is nonzero. If $\pi$ annihilates $B\cap J$ then $\pi$ defines an irreducible representation of the quotient $B/(B\cap J)$. We showed above that $B/(B\cap J)=A/J$, so $\pi$ must annihilate $J$. We assume that $\pi$ does not annihilate $J$, and so $\pi$ cannot annihilate $B\cap J$. Thus $B\cap J$ has property (1) from the definition of `separating'. To prove that it also has property (2), suppose that $\pi,\rho$ are irreducible representations of $A$ whose restrictions to $J$ are nonzero, and whose restrictions to $B\cap J$ are equivalent. Then $\pi\restrict_B$ is equivalent to $\rho\restrict_B$, and since $B$ is separating in $A$ we conclude that $\pi$ and $\rho$ are equivalent as representations of $A$, whence also of $J$.

The `deduce' part follows by induction on $n$: if $n=1$ then $A=A/0$ has the Stone--Weierstrass property and there is nothing to prove. For the inductive step we have a short exact sequence $0\to J_1\to A \to A/J_1\to 0$, where $A/J_1$ has the SWP (by hypothesis) and $J_1$ has the SWP (by the induction hypothesis), so $A$ also has the SWP by what we proved above.
\end{proof}

\begin{proof}[Exercise \ref{ex:SW-plane-example}]
Let $J_0=A$. Let $\phi_0:J_0\to \Mat_2$ be the map $f\mapsto f(0)$, and let $J_1\coloneqq \ker \phi_0$. The image of $\phi_0$ is $\C \id_2$, the scalar multiples of the identity matrix, because these are the only matrices that commute with every $w\in W$. Thus $J_0/J_1\cong \C$, which certainly has the SWP.

Next consider the sets
\[
X = \{(x,0)\ |\ x>0\} \quad \text{and}\quad Z = \{(x,x)\ |\ x>0\}.
\]
For each $x\in X$ we have $W_x=\{1,s\}$, while for each $z\in Z$ we have $W_z = \{1,t\}$. Let $\phi_1:J_1\to C_0(X,\Mat_2)\oplus C_0(Z,\Mat_2)$ be the map $\phi_1(f)=(f\restrict_X,f\restrict_Z)$, and let $J_2\coloneqq \ker \phi_1$.  Since $W_x=\langle s\rangle$ for each $x\in X$, and $W_z=\langle t\rangle$ for each $z\in Z$, we have
\[
\phi_1(J_1) = C_0(X, s') \oplus C_0(Z, t')
\]
where $s'=\{k\in \Mat_2\ |\ ks=sk\} = \left\{\smallbmat{a & 0 \\ 0 & b}\right\}$ and $t' = \{k\in \Mat_2\ |\ kt=tk\} = \left\{\smallbmat{a & b \\ b & a}\right\}$. The $C^*$-algebras $s'$ and $t'$ are both isomorphic to $\C\oplus\C$, so $J_1/J_2\cong C_0(X,s')\oplus C_0(Z,t')$ is a direct sum of copies of $C_0(\R)$, hence has the SWP. 

Now consider the set
\[
P = \{(x,y)\in \R^2\ |\ x>0,\ 0<y<x\}.
\]
We have $W_x=\{1\}$ for each $x\in P$. Restriction of functions thus gives an isomorphism of $C^*$-algebras
\(
\phi_2: J_2 \to C_0(P,\Mat_2).
\)
 We know that $C_0(P,\Mat_2)$ has the Stone--Weierstrass property, so taking $J_3=0$, the composition series $A=J_0\supset J_1\supset J_2\supset 0$ has the required properties. Exercise \ref{ex:SW-extensions} now implies that $A$ has the SWP.
\end{proof}

\begin{proof}[Exercise \ref{ex:full}]
Part (1) follows easily from the properties of the inner product. For (2): let $J\subseteq C_0(X)$ be the ideal $\overline{\lspan}\{\langle \xi\, |\, \eta\rangle\ |\ \xi,\eta\in C_0(X,H)\}$. For each $x\in X$ we can find a function $j\in J$  that does not vanish at $x$, by letting $\xi\in C_0(X,H)$ be a function that does not vanish at $x$ and then taking $j=\langle\xi\,|\,\xi \rangle$. Now Theorem \ref{thm:ideal-irreps} part (7) ensures that $J=C_0(X)$.
\end{proof}

\begin{proof}[Exercise \ref{ex:V-tensor-E}]
For part (1): Conditions (1)--(3) of Definition \ref{def:Hilbert-module} are easy to check. To check conditions (4) and (5), note that the characterisation of positivity in terms of representations given in part (a) of Definition \ref{def:states} implies that if $a$ and $b$ are positive elements of a $C^*$-algebra, then $a+b$ is positive, and $\|a+b\|\geq \max\{\|a\|,\|b\|\}$. 

For part (2): choosing an orthonormal basis for $V$ identifies $V\otimes E$ with the direct sum (in the sense of part (1)) of $\dim V$ copies of $E$.
\end{proof}

\begin{proof}[Exercise \ref{ex:non-adjointable}]
Denote the inclusion map by $t$, and suppose that this map has an adjoint $t^*:C([0,1])\to C_0((0,1])$. The adjunction relation implies that $t^*$ maps the constant function $1\in C([0,1])$ to a multiplicative identity for the $C^*$-algebra $C_0((0,1])$; but this algebra has no multiplicative identity.
\end{proof}

\begin{proof}[Exercise \ref{ex:K-C_0-H}]
Each  $k\in C_0(X,\Compact(H))$ defines a linear operator on $C_0(X,H)$ by pointwise multiplication, and these operators are all adjointable (via the pointwise adjoint). In this way we can view $C_0(X,\Compact(H))$ as a subalgebra of the $C^*$-algebra of adjointable operators $\Adjointable_{C_0(X)}(C_0(X,H))$. 

For all  $\xi,\eta\in C_0(X,H)$, the operator $\ket{\eta}\bra{\xi}\in \Compact_{C_0(X)}(C_0(X,H))$ is pointwise multiplication by the function $(x\mapsto \ket{\eta(x)}\bra{\xi(x)}) \in C_0(X,\Compact(H))$, and so $\Compact_{C_0(X)}(C_0(X,H))\subseteq C_0(X,\Compact(H))$. To prove that we have equality, note that $\Compact_{C_0(X)}(C_0(X,H))$ is an ideal in $C_0(X,\Compact(H))$, so we just need to show (thanks to Theorem \ref{thm:ideal-irreps}(7)) that each irreducible representation of $C_0(X,\Compact(H))$ is nonzero on this ideal. The irreducible representations in question are just the point-evaluations. For each $x\in X$ we  let $\xi\in C_0(X,H)$ be a function with $\xi(x)\neq 0$, and then we get an operator $\ket{\xi}\bra{\xi}\in \Compact_{C_0(X)}(C_0(X,H))$ that is nonzero at $x$.
\end{proof}

\begin{proof}[Exercise \ref{exercise:E-dual}]
For (1): According to the definition of compact operators we have $\Compact_B(E,B) = \overline{\lspan}\{ \ket{b}\bra{\xi}\ |\ b\in B,\ \xi\in E\}$. We have on the one hand $\ket{b}\bra{\xi}=\bra{\xi b^*}$; while on the other hand we know (from Lemma \ref{lem:E=EB}) that each $\xi\in E$ can be written as $\eta b$ for some $\eta\in E$ and $b\in B$, giving $\bra{\xi}=\ket{b^*}\bra{\eta}$. Lastly check that $\{ \bra{\xi}\ |\ \xi\in E\}$ is already an operator-norm-closed linear subspace of $\Adjointable_B(E,B)$.

For (2): First check that $b\mapsto m_b$ is a homomorphism of $C^*$-algebras $B\to \Adjointable_B(B)$; then use part (1) with $E=B$ to conclude that this homomorphism is a surjection $B\to\Compact_B(B)$. Finally, observe that this homomorphism is injective by noting that $\|m_b(b^*)\|=\|b\|^2$.

For (3): It is easy to check that $\llangle\ |\ \rrangle$ satisfies properties (1)--(3) in Definition \ref{def:Hilbert-module}. To check properties (4) and (5), we first note that for each $k\in \Compact_B(E,B)\subset \Compact_B(E\oplus B)$ the operator $\langle k\, |\, k\rangle = k^*k\in \Compact_B(E)$ is clearly a positive element of $\Compact_B(E\oplus B)$, and the characterisation of positivity in terms of representations (from Definition \ref{def:states}) ensures that if an element $a$ of a $C^*$-algebra $A$ is positive in $A$ and lies in a subalgebra $B$, then $a$ is also positive as an element of $B$. Moreover, we have 
\[
\|\langle k\, |\, k\rangle\|_{\Compact_B(E)}= \|\langle k\, |\, k\rangle\|_{\Compact_B(E\oplus B)} = \|k^*k\|_{\Compact_B(E\oplus B)} = \|k\|^2_{\Compact_B(E\oplus B)} = \|k\|^2_{\Compact_B(E,B)}.
\]
Since we know that $\Compact_B(E,B)$ is a Banach space in the operator norm, this completes the proof that $\llangle\ |\ \rrangle$ makes $\Compact_B(E,B)$ into a Hilbert $\Compact_B(E)$-module. This module is full because for all $\xi,\eta\in E$ we have $\llangle\, \bra{\eta}\, |\, \bra{\xi}\, {\rrangle} = \ket{\eta}\bra{\xi}$, and $\Compact_B(E)$ is (by definition) densely spanned by the $\ket{\eta}\bra{\xi}$s. This concludes the proof of part (3).

For part (4): for $\xi,\eta\in E$, the operator $\big| \bra{\xi} \big\rangle \big\langle \bra{\eta}\big|\in \Compact_{\Compact_B(E)}(\Compact_B(E,B))$ sends a compact operator $k\in \Compact_B(E,B)$ to the operator $\bra{\xi} \llangle \, \bra{\eta}\, |\, k \rrangle = \bra{\xi}\circ \ket{\eta}\circ k = m_{\langle \xi\, |\, \eta\rangle}\circ k$. This computation shows, on the one hand, that each $\Compact_B(E)$-compact operator on $\Compact_B(E,B)$ has the form $m_b$. On the other hand, since $E$ is full over $B$, this computation also shows that each $m_b$ is a compact operator on $\Compact_B(E,B)$. So the map $b\mapsto m_b\circ$ is a surjection from $B$ to $\Compact_{\Compact_B(E)}(\Compact_B(E,B))$. It is straightforward to check that this map is a homomorphism of $C^*$-algebras. To see that this map is an isomorphism, suppose that $m_b\circ k=0$ for all $k\in \Compact_B(E,B)$. Then for all $\xi,\eta\in E$ we have $b\langle \xi\, |\, \eta\rangle = m_b\circ \bra{\xi}(\eta) = 0$. Since $E$ is full over $B$ the elements $\langle \xi\, |\, \eta\rangle$ span a dense subspace of $B$, so $ba=0$ for every $a\in B$. Then $\|b\|^2 = \|bb^*\|=0$, so $b=0$.

For part (5): the fact that the given inner product makes $\Compact_B(B,E)$ into a right Hilbert module over $\Compact_B(B)$ follows, as in part (4), from embedding $\Compact_B(B,E)$ into the $C^*$-algebra $\Compact_B(E\oplus B)$. The isomorphism $\Compact_B(B)\cong B$ comes from part (2). Since each $\xi\in E$ can be written as $\eta b$ for some $\eta\in E$ and some $b\in B$, the operator $\ket{\xi}=\ket{\eta}\bra{b^*}$ is $B$-compact. The same equality shows that $\Compact_B(B,E)$ is densely spanned by the $\ket{\xi}$s. Straightforward computations verify that the map $\xi\mapsto \ket{\xi}$ is an inner-product-preserving (hence, in particular, isometric) map of Hilbert $B$-modules, and since the image of this isometric map is dense in $\Compact_B(B,E)$ the map is surjective.
\end{proof}

\begin{proof}[Exercise \ref{ex:crossed-product-norm}]
Compute! For example, to prove that $t_{bw}$ is adjointable, with $t_{bw}^* = t_{(bw)^*}$, we compute
\begin{align*}
& \langle \delta_{v_1}\otimes a_1\, |\, t_{bw}(\delta_{v_2}\otimes a_2) \rangle  = \langle \delta_{v_1}\otimes a_1\, |\, \delta_{wv_2}\otimes \beta_{v_2^{-1}w^{-1}}(b)a_2\rangle \\
& = \langle \delta_{v_1}\, |\, \delta_{wv_2}\rangle a_1^*\beta_{v_2^{-1}w^{-1}}(b)a_2 = \langle \delta_{w^{-1}v_1}\, |\, \delta_{v_2}\rangle \left( \beta_{v_1^{-1}}(b^*) a_1\right)^* a_2 \\
& = \langle \delta_{w^{-1}v_1} \otimes \beta_{v_1^{-1}w}(\beta_{w^{-1}}(b^*))a_1\, |\, \delta_{v_2}\otimes a_2\rangle 
= \langle t_{(bw)^*}(\delta_{v_1}\otimes a_1)\, |\, \delta_{v_2}\otimes a_2\rangle.\qedhere
\end{align*}
\end{proof}

\begin{proof}[Exercise \ref{ex:Morita-symmetric}]
Let $E$ be a Morita equivalence from $A$ to $B$. Using part (3) of Exercise \ref{exercise:E-dual}, and the isomorphism $A\cong \Compact_B(E)$ that is part of the data of the given Morita equivalence, we can make $\Compact_B(E,B)$ into a full right Hilbert $A$-module. Part (4) of Exercise \ref{exercise:E-dual} then gives an isomorphism $B\cong \Compact_{A}(\Compact_B(E,B))$, so $\Compact_B(E,B)$ is a Morita equivalence from $B$ to $A$.
\end{proof}

\begin{proof}[Exercise \ref{ex:Morita-sum}]
For each $i$ let $E_i$ be a Morita equivalence from $A$ to $B$, with isomorphisms $\phi_i:A_i\to \Compact_{B_i}(E_i)$. Form the $c_0$-direct sum
\[
E =\bigoplus_i E_i \coloneqq \{ (\xi_i)_{i\in I}\ |\ \|\xi_i\|\to 0\text{ as }i\to\infty\},
\]
and check that the definition $\langle (\xi_i)_i\, |\, (\eta_i)_i\rangle = (\langle \xi_i\, |\, \eta_i)_i$ makes $E$ into a Hilbert module over $B=\bigoplus_i B_i$, for which the map $\bigoplus_i \phi_i : \bigoplus_i A_i \to \bigoplus_i \Compact_{B_i}(E_i) =\Compact_B(E)$ is an isomorphism.
\end{proof}

\begin{proof}[Exercise \ref{ex:W-acts-on-K}]
First use the equalities in Definition \ref{def:equivariant-Hilbert-module} to show that for all $\xi,\eta,\zeta\in E$ and $w\in W$ we have $\gamma_w \ket{\xi}\bra{\eta} \gamma_w^{-1}\zeta = \ket{\gamma_w \xi}\bra{\gamma_w\eta} \zeta$. In particular, $\gamma_w \ket{\xi}\bra{\eta} \gamma_w^{-1}$ lies in $\Compact_B(E)$, and so by linearity and continuity we have $\alpha_w(k)\in \Compact_B(E)$ for all $k\in \Compact_B(E)$. It is easy to see that each $\alpha_w$ is an algebra automorphism, and the equality $\alpha_w(\ket{\xi}\bra{\eta})=\ket{\gamma_w\xi}\bra{\gamma_w\eta}$ implies that we also have $\alpha_w(k^*)=\alpha_w(k)^*$ for all $k\in \Compact_B(E)$. The fact that $\alpha_{w_1}\alpha_{w_2}=\alpha_{w_1w_2}$ follows from the assumption that $\gamma_{w_1}\gamma_{w_2}=\gamma_{w_1w_2}$.
\end{proof}

\begin{proof}[Exercise \ref{ex:C-ideal}]
For each $x\in X$ we define $C(Wx,W,I)\subseteq C(Wx)\rtimes W$, as in Definition \ref{def:CXWI}, by restricting our attention to the operators $I_{w,y}$ for $y$ in the orbit $Wx$. Restriction of functions from $X$ to $Wx$ gives a homomorphism of $C^*$-algebras $\pi_{Wx}:C_0(X)\rtimes W\to C(Wx)\rtimes W$, and since $C(X,W,I)$ is defined by pointwise conditions we have
\[
C(X,W,I) = \bigcap_{x\in X} \pi_x^{-1}(C(Wx, W,I)).
\]
Preimages and intersections of ideals are ideals, so it is enough to prove that $C(Wx,W,I)$ is an ideal in $C(Wx)\rtimes W$. Since $C(Wx)\rtimes W$ is finite-dimensional we don't have to worry about checking closure in the norm, and routine algebraic computations confirm that $C(Wx,W,I)$ is closed under multiplication on either side by each $fw\in C(Wx)\rtimes W$.
\end{proof}

\bibliographystyle{alpha}
\bibliography{CIRM.bib}

\end{document}